\documentclass[12pt]{amsart}

\usepackage{accents}
\usepackage{enumerate}
\usepackage{color}
\usepackage{graphicx}

\usepackage{mathrsfs}
\usepackage{mathtools 
}

\usepackage[
pdftex,hyperfootnotes]{hyperref}

\usepackage{nicematrix}

\usepackage{pgfplots}
\usepackage{tikz}
\usepackage{tikz-3dplot}
\usetikzlibrary{calc,shadows,shapes.callouts,shapes.geometric,shapes.misc,positioning,patterns,decorations.pathmorphing,decorations.markings,decorations.fractals,decorations.pathreplacing,shadings,fadings}

\usepackage
{xcolor}

\usepackage[ansinew]{inputenc}
\usepackage[T1]{fontenc}

\usepackage{textcomp}
\usepackage{charter}
\usepackage[]{bera}
\usepackage[charter]{mathdesign}

\usepackage{longtable}
\usepackage{enumerate}
\allowdisplaybreaks

\newcommand{\lvt}{\left|\kern-1.35pt\left|\kern-1.3pt\left|}
\newcommand{\rvt}{\right|\kern-1.3pt\right|\kern-1.35pt\right|}

\newtheorem{thm}{Theorem}[section]

\theoremstyle{remark}

\def\lla{\langle{\kern-2.5pt}\langle} 
\def\rra{\rangle{\kern-2.5pt}\rangle}

\begin{document}
 
\title[Matrix Orthogonal Polynomials: A~Riemann--Hilbert approach]{Matrix Orthogonal Polynomials: \\
A~Riemann--Hilbert approach
 \\[2ex] 
{\normalfont
\resizebox{.395\hsize}{!}{
{\emph{\small{\MakeLowercase{in memory of} J\MakeLowercase{osé} C\MakeLowercase{arlos} P\MakeLowercase{etronilho}}}}}}
}


\author
{Am\'\i lcar Branquinho}
\address{CMUC,
Department of Mathematics,
University of Coimbra, 3000-143 Coimbra,
Portugal}
\email{ajplb@mat.uc.pt}

\author
{Ana Foulqui\'e-Moreno}
\address{CIDMA,
Departamento de Matem\'atica,
Universidade de Aveiro, 3810-193 Aveiro,
Portugal}
\email{foulquie@ua.pt}

\author
{Assil Fradi}
\address{Mathematical Physics Special Functions and Applications Laboratory,
The \linebreak
Higher School of Sciences and Technology of Hammam Sousse,
University of Sousse, Sousse 4002,
Tunisia}
\email{assilfradi@ua.pt}

\author
{Manuel Ma\~nas}
\address{Departamento de F\'\i sica Te\'orica,
Universidad Complutense de Madrid, Plaza Ciencias 1, 28040-Madrid,
Spain}
\email{manuel.manas@ucm.es}


\date{\today} 

 
\begin{abstract}
In this work we show how to get advantage from the Rie\-mann--Hilbert analysis in order to obtain information about the matrix orthogonal polynomials and functions of second kind associated with a weight matrix. We deduce properties for the recurrence relation coefficients from differential properties of the weight matrix. We take the matrix polynomials of Hermite, Laguerre and Jacobi type as a case~study.
\end{abstract}

\maketitle

\section*{Introduction} \label{sec:0}

The main purpose of this paper is to show how the Riemann--Hilbert analysis could help in order to derive algebraic and analytic properties for the sequences of matrix orthogonal polynomials associated with a weight matrix in the class of Hermite, Laguerre or Jacobi, presented in a sequel of works~\cite{BAFMM,BFAM,PAMS_AB_AF_MM}. 
Hence we will take that works as a basis and the references therein.

Since Krein presented a matrix extension of real orthogonal polynomials in 1949 (cf.~\cite{Krein1,Krein2}), several authors considered these type of families~\cite{nikishin,Berezanskii,geronimo}. 
Also, several examples have been considered of matrix-valued orthogonal polynomials, most of them associated with second order differential equations (cf.~\cite{grummmm,grummmmpatirao1}.

The study of equations for the recursion coefficients for orthogonal polynomials constitutes a subject of current interest. The question of how some properties of the weight, for example to satisfy a Pearson type equation,
translates to the recursion coefficients has been treated in several places, for a review see~\cite{VAssche}. 
Here we will see that these properties can be derived directly from a generalized matrix Pearson differential equation.

Now, we present the structure of the work:
In Section~\ref{sec:1} we state the basic facts of the theory of matrix orthogonal polynomials that will be used in the text. 
We begin with the definition of regular weight matrix and the three term recurrence relation for the matrix orthogonal polynomials. Next we reinterpret the Berezanskii matrix orthogonal polynomials in terms of scalar orthogonal polynomials.
We also make a reinterpretation of the matrix orthogonality in terms of the Gauss--Borel factorization of the moment matrix, and at the end of the section we state a general Riemann--Hilbert problem associated with the sequences of monic matrix orthogonal polynomials.
In Section~\ref{sec:2} we present the general weights that we have considered in the previous works on matrix orthogonal polynomials. We will see that all the properties of these systems come from the analytical properties of a fundamental matrix associated with the matrix functions $Y_n^{\mathsf L}$ (as well as~$Y_n^{\mathsf R}$) and of the weight matrix, $\omega$, that will be called $M_n^\mathsf L$, $M_n^\mathsf R$.
As the weight functions considered here satisfy a generalized Pearson matrix equation we could present in Section~\ref{sec:2} a first order differential equations for the $Y_n^{\mathsf L}$ and~$Y_n^{\mathsf R}$.
From these first order differential equations we derive a second order ones for the $Y_n^{\mathsf L}$ and $Y_n^{\mathsf R}$ respectively. This is the subject of Section~\ref{sec:3}.
There, we will construct second order differential operators that have sequence of Berezanskii matrix polynomials (respectively, of second kind matrix functions) as eigenfunctions. 
We end this work, with Section~\ref{sec:4}, by showing some examples of discrete matrix Painlev\'e equations for the three term recurrence relation coefficients associated with generalized Hermite, Laguerre and Jacobi type weight matrix.

\section{Matrix biorthogonality} \label{sec:1}

Let
\begin{align*}
 \omega
=
\begin{bmatrix}
\omega^{(1,1)} & \cdots & \omega^{(1,N)} \\
 \vdots & \ddots & \vdots \\
\omega^{(N,1)} & \cdots & \omega^{(N,N)}
\end{bmatrix}
 ,
\end{align*}
 be a $\displaystyle N \times N $ weight matrix with support on the real line,~$\displaystyle \mathbb R $ or on a smooth oriented non self-intersecting curve~$\gamma$, i.e. $\displaystyle \omega \in \mathbb M_N (\mathbb C) $ and $\displaystyle \omega^{(j,k)} $ is, for each $\displaystyle j,k \in \big\{ 1, \ldots , N \big\} $, a complex weight with support on the real line or on a smooth oriented non self-intersecting 
curve~$\gamma$, 
in the complex plane $\mathbb C$, i.e.
 $ \omega^{(j,k)} $ 
is, for each $j,k \in \big\{ 1, \ldots , N \big\}$, a complex weight with support on $\gamma$.
We define the moment of order~$\displaystyle n $ associated with~$\displaystyle \omega $~as 
\begin{align*}
\omega_n = \int_\gamma x^n \, \omega (x) \, \frac{\operatorname d x}{2\pi \operatorname i} , && n \in \mathbb N
\coloneqq \big\{ 0,1, \ldots \big\} .
\end{align*}
We say that $ \omega $ is \emph{regular} if the moments, $\omega_n$, $n \in \mathbb N$, exist and the $n$-th matrix of~moments, 
\begin{align*}
\pmb{\mathsf U}_n 
 \coloneqq 
\begin{bmatrix}
\omega_{0} & 
\cdots 
& \omega_n \\
\vdots 
& \ddots 
& \vdots \\
\omega_{n} 
& \cdots
& \omega_{2n} 
\end{bmatrix} , && n \in \mathbb N ,
\end{align*}
is such that
\begin{align} \label{eq:regular}
\det \pmb{\mathsf U}_n \not = 0 , && n \in \mathbb N.
\end{align}
In this way, we define a \emph{sequence of matrix monic polynomials},
 $\big\{ {P}_n^{\mathsf L} \big\}_{n\in \mathbb N} $, where $\deg {P}_n^{\mathsf L} (z) = n$, $n \in \mathbb N$, \emph{left orthogonal} and \emph{right orthogonal}, $\big\{ P_n^{\mathsf R} \big\}_{n\in\mathbb N} $, where $\deg {P}_n^{\mathsf R} (z) = n$, $n \in \mathbb N$, with respect to a regular weight matrix $\omega$, by the conditions,
\begin{align}
\label{eq:ortogonalidadL}
 \int_\gamma {P}_n^{\mathsf L} (z) \omega (z) z^k \frac{\operatorname d z}{2\pi\operatorname i}
 = \delta_{n,k} C_n^{-1} , \\
 \label{eq:ortogonalidadR}
 \int_\gamma z^k \omega (z) {P}_n^{\mathsf R} (z) \frac{\operatorname d z}{2\pi\operatorname i}
 = \delta_{n,k} C_n^{-1} , 
\end{align}
for $k \in \big\{ 0, 1, \ldots , n \big\}$ and $n \in \mathbb N$,
where $C_n $ is a nonsingular matrix.

We can see (cf.~\cite{BAFMM}) that a sequence of monic polynomials $\big\{ P_n^{\mathsf L} \big\}_{n \in \mathbb N} $, respectively, $\big\{ P_n^{\mathsf R} \big\}_{n \in \mathbb N} $, is defined by~\eqref{eq:ortogonalidadL}, respectively,~\eqref{eq:ortogonalidadR}, with respect to a regular weight matrix,~$\omega$.

Notice that neither the weight matrix is requested to be Hermitian 
nor the curve~$\gamma$ to be on the real line, i.e., we are dealing, in principle with nonstandard orthogonality and, consequently, with biorthogonal matrix polynomials instead of orthogonal matrix polynomials, i.e.
the sequences of matrix polynomials $\big\{ P_n^{\mathsf L} \big\}_{n \in \mathbb N }$ and $\big\{ P_n^{\mathsf R} \big\}_{n \in \mathbb N }$ are biorthogonal with respect to a weight matrix functions $\omega$, as from~\eqref{eq:ortogonalidadL} and~\eqref{eq:ortogonalidadR}
\begin{align} \label{eq:biorthogonality}
\int_\gamma P_n^{\mathsf L} (t) \omega (t) {P}_m^{\mathsf R} (t) \frac{\operatorname d t}{ 2 \pi \operatorname i}
 = \delta_{n,m} C_n^{-1}, && n , m \in \mathbb N .
\end{align}
 As the polynomials are chosen to be monic, we can write
\begin{align*}
 {P}_n^{\mathsf L} (z) & =\operatorname I z^n + p_{\mathsf L,n}^1 z^{n-1} + p_{\mathsf L,n}^2 z^{n-2} + \cdots + p_{\mathsf L,n}^n , 
 \\
{P}_n^{\mathsf R} (z) & =\operatorname I z^n + p_{\mathsf R,n}^1 z^{n-1} + p_{\mathsf R,n}^2 z^{n-2} + \cdots + p_{\mathsf R,n}^n , 
\end{align*}
with matrix coefficients $p_{\mathsf L, n}^k , p_{\mathsf R, n}^k \in 
\mathbb C^{N\times N} $, $k = 0, \ldots, n $ and $n \in \mathbb N $ (imposing that $p_{\mathsf L,n}^0 = p_{\mathsf R,n}^0= \operatorname I $, $n \in \mathbb N $). Here $\operatorname I\in\mathbb C^{N\times N}$ denotes the identity matrix.

From~\eqref{eq:ortogonalidadL} we deduce that the Fourier coefficients of the expansion
\begin{align*}
z {P}^{\mathsf L}_n (z) = \sum_{k=0}^{n+1} \ell_{\mathsf L,k}^n {P}^{\mathsf L}_k (z) ,
\end{align*}
are given by $\ell_{\mathsf L,k}^n =\pmb 0 $, $k = 0 , 1, \ldots , n-2 $ (here we denote the zero matrix by~$\pmb 0$),
 $\ell_{\mathsf L,n-1}^n = C_n^{-1} C_{n-1} $ (is a direct consequence of orthogonality conditions), $\ell_{\mathsf L,n+1}^n = \operatorname I $ (as $ {P}^{\mathsf L}_n (z) $ are monic polynomials) and 
 $\ell_{\mathsf L,n}^n = 
p_{\mathsf L,n}^1 - p_{\mathsf L,n+1}^1 $ $\eqqcolon
\beta^\mathsf L_n $
(by comparison of the coefficients, assuming $C_0 = \operatorname I$). 

Hence, assuming the orthogonality relations~\eqref{eq:ortogonalidadL}, we conclude that the sequence of monic polynomials $\big\{ {P}^{\mathsf L}_n \big\}_{n\in\mathbb N } $ is defined by the three term recurrence relation
\begin{align} \label{eq:ttrr}
z {P}^{\mathsf L}_n (z) = {P}^{\mathsf L}_{n+1} (z) + \beta^{\mathsf L}_n {P}^{\mathsf L}_{n} (z) + \gamma^{\mathsf L}_n {P}^{\mathsf L}_{n-1} (z), && n \in \mathbb N , 
\end{align}
with recursion coefficients 
\begin{align*}
\beta^\mathsf L_n& \coloneqq
p_{\mathsf L,n}^1 - p_{\mathsf L,n+1}^1, &
\gamma^{\mathsf L}_{n+1} 
 &\coloneqq C_{n+1}^{-1} C_{n},
 && n \in \mathbb N ,
\end{align*}
 and initial conditions, $ {P}^{\mathsf L}_{-1} =\pmb 0 $ and 
 $ {P}^{\mathsf L}_{0} = \operatorname I $.

Any sequence of monic matrix polynomials, $\big\{ P^{\mathsf R}_{n} \big\}_{n\in\mathbb N } $, with $\deg P^{\mathsf R}_n= n $, biorthogonal with respect to $\big\{ P_n^{\mathsf L} \big\}_{n\in\mathbb N }$ and $ \omega (z)$, i.e.~\eqref {eq:biorthogonality}
is fulfilled, also satisfies a three term relation.
To prove this we proceed in the same way as in the left case, arriving to
$P^{\mathsf R}_{-1} =\pmb 0 $, $P^{\mathsf R}_{0} = \operatorname I$,
\begin{align} 
 & z P^{\mathsf R}_n (z) = P^{\mathsf R}_{n+1} (z) + P^{\mathsf R}_{n} (z) \beta^{\mathsf R}_n + P^{\mathsf R}_{n-1} (z) \gamma^{\mathsf R}_n ,& n &\in\mathbb N ,
\label{eq:rightttrr}
\end{align}
where 
\begin{align}
\label{eq:rightcoeff}
\beta_n^{\mathsf R} &\coloneqq C_n \beta^{\mathsf L}_n C_n^{-1}, &\gamma^{\mathsf R}_n &\coloneqq C_n \gamma^{\mathsf L}_nC_n^{-1} = C_{n-1} C_n^{-1},
\end{align}
and the orthogonality conditions~\eqref{eq:ortogonalidadR} are satisfied.

\subsection{Berezanskii matrix orthogonal polynomials}\label{sec:12}

\subsubsection{First example}

The notion of matrix orthogonality can be seen coming from the scalar one. In fact, following the ideas of Berezanskii, in~\cite{Berezanskii}, given two sequences of monic polynomials, $\big\{ p_n^1 \big\}_{n \in \mathbb N}$ and $\big\{ p_n^2 \big\}_{n \in \mathbb N}$, orthogonal with respect to $\omega^1$, $\omega^2$, respectively, with the same support of orthogonality, $\Omega \subset \mathbb R$,~i.e.
\begin{align*}
\int_\Omega p_n^1 (t) \omega^1 (t) p_m^1 (t) \, \frac{\operatorname d t}{2\pi\operatorname i}
&= \kappa_n^1 \, \delta_{n,m} , 
 \\
\int_\Omega p_n^2 (t) \omega^2 (t) p_m^2 (t) \, \frac{\operatorname d t}{2\pi\operatorname i}
&= \kappa_n^2 \, \delta_{n,m} , && n , m \in \mathbb N ,
\end{align*}
with $\kappa_n^1 , \kappa_n^2 
>0$, $n \in \mathbb N$; 
we can construct a matrix sequence of monic polynomial
\begin{align}
\label{eq:diagonalpol}
\mathbb P_n (x) = \frac 1 2
\begin{bmatrix} p_n^1 (x) + p_n^2 (x) & p_n^1 (x) - p_n^2 (x) \\[.1cm]
p_n^1 (x) - p_n^2 (x) & p_n^1 (x) + p_n^2 (x)
\end{bmatrix}, && n \in \mathbb N ,
\end{align}
orthogonal with respect to the weights matrix in $\Omega$
\begin{align*}
\mathbb W (x) = \frac 1 2
\begin{bmatrix} \omega^1 (x) + \omega^2 (x) & \omega^1 (x) - \omega^2 (x) \\[.1cm]
\omega^1 (x) - \omega^2 (x) & \omega^1 (x) + \omega^2 (x) \end{bmatrix}.
\end{align*}
In fact,
\begin{align}
\label{eq:ortexemplo}
\int_\Omega
\mathbb P_n (t)
\mathbb W (t)
\mathbb P_m^\top (t)
\, \frac{\operatorname d t}{2 \pi \operatorname i}
=
\mathbb K_n
 \, \delta_{n,m},
 && n,m \in \mathbb N ,
\end{align}
where $\mathbb K_n$ is an invertible matrix given by
\begin{align*}
\mathbb K_n \coloneqq \frac 1 2
\begin{bmatrix}
\kappa_n^1 + \kappa_n^2 & \kappa_n^1 - \kappa_n^2 \\[.1cm]
\kappa_n^1 - \kappa_n^2 & \kappa_n^1 + \kappa_n^2 \end{bmatrix} ,
 && n \in \mathbb N .
\end{align*}
This example can be deconstructed in order to get the scalar orthogonality.
In fact, we can rewrite $\mathbb P_n$ and $\mathbb W$ as
\begin{align*}
\mathbb P_n (x) &
= \frac 1 2 \, \pmb \alpha \, 
\begin{bmatrix}
p_n^1 (x) & 0 \\
0 & p_n^2 (x) 
\end{bmatrix}
 \,
\pmb \alpha^\top
 , && n \in \mathbb N , \\
\mathbb W (x) &
= \frac 1 2
\, \pmb \alpha \,
\begin{bmatrix}
\omega^1 (x) & 0 \\
0 & \omega^2 (x) 
\end{bmatrix}
\, \pmb \alpha^\top
 ,
\end{align*}
with
\begin{align}\label{eq:inv}
\pmb \alpha =
\begin{bmatrix}
1 & -1 \\ 1& 1
\end{bmatrix}, &&
\pmb \alpha^{-1} =
\frac 1 2
\begin{bmatrix}
1 & -1 \\ 1& 1
\end{bmatrix}^\top .
\end{align}
Now, taking into account that~\eqref{eq:inv}
equation~\eqref{eq:ortexemplo} takes the form
\begin{multline*}
\int_\Omega
\mathbb P_n (t)
\mathbb W (t)
\mathbb P_m^\top (t)
\, \frac{\operatorname d t}{2 \pi \operatorname i}
 \\ =
\frac 1 2 \,
\pmb \alpha \,
 \int_\Omega
\begin{bmatrix}
p_n^1 (t) & 0 \\
0 & p_n^2 (t) 
\end{bmatrix}
\begin{bmatrix}
\omega^1 (t) & 0 \\
0 & \omega^2 (t) 
\end{bmatrix}
\begin{bmatrix}
p_m^1 (t) & 0 \\
0 & p_m^2 (t) 
\end{bmatrix}^\top
\, \frac{\operatorname d t}{2 \pi \operatorname i}
\, \pmb \alpha^\top
,
\end{multline*}
and so
\begin{align*}
\int_\Omega
\mathbb P_n (t)
\mathbb W (t)
\mathbb P_m^\top (t)
\, \frac{\operatorname d t}{2 \pi \operatorname i}
 =
\frac 1 2
 \, \pmb \alpha \, 
\begin{bmatrix}
\kappa_n^1 & 0 \\ 0 & \kappa_n^2
\end{bmatrix}
\, \pmb \alpha^\top
\, \delta_{n,m}
 , && n,m \in \mathbb N
 .
\end{align*}
Now, applying once time~\eqref{eq:inv}, we get from the last equation
\begin{align*}
\int_\Omega
\begin{bmatrix}
p_n^1 (t) \omega^1 (t) p_m^1 (t) & 0 \\
 0 & p_n^2 (t) \omega^2 (t) p_m^2 (t)
\end{bmatrix}
\, \frac{\operatorname d t}{2 \pi \operatorname i}
 =\begin{bmatrix}
\kappa_n^1 & 0 \\ 0 & \kappa_n^2
\end{bmatrix}
 \, \delta_{n,m}
, && n,m \in \mathbb N .
\end{align*}
Hence we recover the initial data written in matrix notation.

When we consider, as in~\cite{Berezanskii}, some specific Jacobi weights in $[-1,1]$, 
\begin{align*}
\omega^1 (x) = \sqrt{\frac{1-x}{1+x}} , && \omega^2 (x) = \sqrt{\frac{1+x}{1-x}} , && x \in [-1,1] .
\end{align*}
In the notation just given, we have
\begin{align*}
\mathbb P_n (x) & = 
\frac{1}{2^n}
\begin{bmatrix}
U_n (x) & - U_{n-1} (x) \\[.1cm]
 - U_{n-1} (x) & U_n (x)
\end{bmatrix} , && n \in \mathbb N , \\
\mathbb W (x) & =
\begin{bmatrix}
\sqrt{1-x^2} & x \, \sqrt{1-x^2} \\[.1cm] x \, \sqrt{1-x^2} & \sqrt{1-x^2} 
\end{bmatrix}
 , && x \in [-1,1] .
\end{align*}
Here the polynomials
$U_n (x) $, $n \in \mathbb N$, are the second kind Chebychev orthogonal polynomials, i.e.
\begin{align*}
U_n (x) = \frac{\sin \big( (n+1) t \big)}{\sin (t)}, && \cos (t) = x, && n \in \mathbb N . 
\end{align*}

\subsubsection{Second example}

It is important to notice that, if we depart from two left and right matrix sequences, $\big\{ P_n^{1,\mathsf L}\big\}$, $\big\{ P_n^{1,\mathsf R}\big\}$, and $\big\{ P_n^{2,\mathsf L}\big\}$, $\big\{ P_n^{2,\mathsf R}\big\}$, biorthogonal with respect to
$W^1$, $W^2$, respectively, with support on the same curve $\gamma \subset \mathbb C$, then the matrix sequence of polynomials
\begin{align*}
\mathbb P_n^{\mathsf L} (x)
 & =
\frac 1 2
\begin{bmatrix}
P_n^{1,\mathsf L} (x) + P_n^{2,\mathsf L} (x)
& P_n^{1,\mathsf L} (x) - P_n^{2,\mathsf L} (x) \\[.1cm]
 P_n^{1,\mathsf L} (x) - P_n^{2,\mathsf L} (x) & 
P_n^{1,\mathsf L} (x) + P_n^{2,\mathsf L} (x)
\end{bmatrix}, 
 \\
\mathbb P_n^{\mathsf R} (x)
 & =
\frac 1 2
\begin{bmatrix}
P_n^{1,\mathsf R} (x) + P_n^{2,\mathsf R}(x) &
P_n^{1,\mathsf R} (x) - P_n^{2,\mathsf R} (x) \\[.1cm]
P_n^{1,\mathsf R} (x) - P_n^{2,\mathsf R} (x) & 
P_n^{1,\mathsf R} (x) + P_n^{2,\mathsf R} (x)
\end{bmatrix}, && n \in \mathbb N ,
\end{align*}
are biorthogonal with respect to $\mathbb W $ defined by
\begin{align*}
\mathbb W (x)
 & =
\frac 1 2
\begin{bmatrix}
W^{1} (x) + W^{2}(x) & W^{1} (x) - W^{2} (x) \\[.1cm]
W^{1} (x) - W^{2} (x) & 
W^{1} (x) + W^{2} (x)
\end{bmatrix}
&& \text{on} && \gamma ,
\end{align*}
i.e. we have
\begin{align*}
\int_\gamma
\mathbb P_n^\mathsf L (t)
\mathbb W (t)
\mathbb P_m^\mathsf R (t)
\, \operatorname d t
 =
\mathbb K_n^{-1} \, \delta_{n,m} , && n,m \in \mathbb N .
\end{align*}
where
\begin{align*}
\mathbb K_n^{-1}
 =
\frac 1 2 
\begin{bmatrix}
\operatorname I &
-\operatorname I \\
\operatorname I & 
\operatorname I
\end{bmatrix}
\begin{bmatrix}
\big( C_n^1 \big)^{-1} &
\pmb 0 \\
\pmb 0 & 
\big( C_n^2 \big)^{-1}
\end{bmatrix}
\begin{bmatrix}
\operatorname I &
-\operatorname I \\
\operatorname I & 
\operatorname I
\end{bmatrix}^\top , && n \in \mathbb N ,
\end{align*}
and $ C_n^1 $, $ C_n^1 $, are invertible matrices coming from
\begin{align*}
\int_\gamma
P_n^{j, \mathsf L} (t)
 W^j (t)
P_m^{j, \mathsf R} ( (t)
\, \operatorname d t
 & =
\big( C_n^j \big)^{-1} \, \delta_{n , m} ,
 && j = 1,2 , && n,m \in \mathbb N .
\end{align*} 
It is important to notice that
\begin{align*}
\begin{bmatrix}
\operatorname I &
-\operatorname I \\
\operatorname I & 
\operatorname I
\end{bmatrix}^{-1}
 = \frac 1 2 
\begin{bmatrix}
\operatorname I &
-\operatorname I \\
\operatorname I & 
\operatorname I
\end{bmatrix}^\top,
\end{align*}
and so we can apply all the procedure explained before in order to reinterpret the matrix orthogonality in the diagonal matrix setting.

\subsection{Gauss--Borel interpretation of the biorthogonality}

\subsubsection{Second kind functions}

We define the \emph{sequence of second kind matrix functions} by
\begin{align} \label{eq:secondkind}
Q^{\mathsf L}_n (z) &
 \coloneqq
 \int_\gamma \frac{P^{\mathsf L}_n (t)}{t-z} { \omega (t)} \frac{\operatorname d t} { 2 \pi \operatorname i} , 
 \\ 
\label{eq:right_secondkind}
{Q}_n^{\mathsf R} (z) &
 \coloneqq \int_\gamma \omega (t) \frac{P^{\mathsf R}_{n} (t)}{t-z} \frac{\operatorname d t }{ 2 \pi \operatorname i},
 && n \in \mathbb N .
\end{align} 
From the orthogonality conditions~\eqref{eq:ortogonalidadL} and~\eqref{eq:ortogonalidadR} we have, for all $n \in \mathbb N $, the fol\-lowing 
asymptotic expansion near infinity for the sequence of functions of the second~kind 
\begin{align*} 
Q^{\mathsf L}_n (z) 
& = - C_n^{-1} \big( \operatorname I z^{-n-1} + q_{\mathsf L,n}^1 z^{-n-2} + \cdots \big),
 \\
Q^{\mathsf R}_n (z) 
& = - \big( \operatorname I z^{-n-1} + q_{\mathsf R,n}^1 z^{-n-2} + \cdots \big) \, C_n^{-1} . 
\end{align*}
From now on we assume that the weights $ W^{(j,k)} $, $j,k \in \big\{ 1, \ldots , N \big\} $ are H\"older continuous.
Hence using the Plemelj's formula, cf.~\cite{gakhov}, 
applied to~\eqref{eq:secondkind} and~\eqref{eq:right_secondkind}, the following fundamental jump identities hold
\begin{align*}
\big( Q^{\mathsf L}_n (z) \big)_+ - \big( Q_n^{\mathsf L} (z) \big)_- &= {P}^{\mathsf L}_n (z) \omega (z), 
 \\
\big( Q^{\mathsf R}_n (z) \big)_+ - \big( Q^{\mathsf R}_n (z) \big)_- &= \omega (z){P}^{\mathsf R}_n (z) ,
 && n \in \mathbb N ,
\end{align*}
$z\in\gamma$, where, $ \big( f(z) \big)_{\pm} = \lim\limits_{\epsilon \to 0^{\pm}} f(z + i \epsilon ) $; 
here $\pm$ indicates the positive/negative region according to the orientation of the curve $\gamma$.

Now, multiplying equation~\eqref{eq:ttrr} on the right by $ \omega$ and integrating we get, using the definition~\eqref{eq:secondkind} of $\big\{ Q^{\mathsf L}_n \big\}_{n\in \mathbb N }$, that
\begin{align*}
\int_\gamma \frac{t \, {P}^{\mathsf L}_n (t)} {t-z} { \omega (t) } \frac{\operatorname d t} {2\pi \operatorname i} = Q^{\mathsf L}_{n+1} (z) + \beta^{\mathsf L}_n Q^{\mathsf L}_{n} (z) + C_n^{-1} C_{n-1} Q^{\mathsf L}_{n-1} (z).
\end{align*}
As $\frac{t}{t-z} = 1 + \frac{z}{t-z} $, from the orthogonality conditons~\eqref{eq:ortogonalidadL} we conclude that
\begin{align}
\label{eq:ttrrql}
z Q^{\mathsf L}_n (z) & = Q^{\mathsf L}_{n+1} (z) 
+ \beta^{\mathsf L}_n Q^{\mathsf L}_{n} (z) + C_n^{-1} C_{n-1} Q^{\mathsf L}_{n-1} (z) , & n \in \mathbb N ,
\end{align}
as well as, from~\eqref{eq:rightttrr}
\begin{align}
\label{eq:ttrrqlr}
z Q^{\mathsf R}_n (z) & = Q^{\mathsf R}_{n+1} (z) + Q^{\mathsf R}_{n} (z) \, \beta^{\mathsf R}_n + Q^{\mathsf R}_{n-1} (z) \, C_{n-1} C_n^{-1} , & n \in \mathbb N ,
\end{align}
with initial conditions 
\begin{align*} 
 & Q^{\mathsf L}_{-1} (z) = Q^{\mathsf R}_{-1} (z) = - C_{-1}^{-1} , \\
 & Q^{\mathsf L}_0 (z)
=Q^{\mathsf R}_0(z)
=S_\omega (z) \coloneqq \int_{\gamma}\frac{ \omega(t)}{t-z}\frac{\operatorname d t}{2\pi \operatorname i} , 
\end{align*}
where $S_\omega (z)$ is the Stieltjes--Markov like transformation of the weight matrix~$ \omega$.


\begin{thm}
\label{pro:biorthogonalitystieltjes}
Let $a$ and $b$ be the end points of $\gamma$, respectively.
Let $C$ be a circle, negatively oriented (clockwise), such that $a$ and $b$ are in the interior of~$C$.
Then the Stieltjes--Markov like transformation $ S_\omega $ is a complex measure of biorthogonality for $\big\{P_n^\mathsf{L}\big\}_{n\in\mathbb N }$ and $\big\{P_n^\mathsf{R}\big\}_{n\in\mathbb N }$ over $C$, i.e.
\begin{align}
\label{eq:ortogonalidadeS}
\int_C P_n^\mathsf{L} (z)S_\omega(z)P_m^\mathsf{R} (z)\, \frac{\operatorname d z}{2\pi \operatorname i}
 =
C^{-1}_n \, \delta_{n,m} , && n , m \in \mathbb N ,
\end{align}for some invertible matrices, $C_n$, $n \in \mathbb N$.
\end{thm}

\begin{proof}
We have the following identities
\begin{align*}
&
\int_C P_n^\mathsf{L} (z)S_\omega(z)P_m^\mathsf{R} (z)\, \frac{\operatorname d z}{2\pi \operatorname i} =\int_C P_n^\mathsf{L} (z) \left(\int_\gamma \frac{\omega (t)}{t-z}\, \frac{\operatorname d t}{2\pi \operatorname i}\right) P_m^\mathsf{R} (z)\, \frac{\operatorname d z}{2\pi \operatorname i} \\
&\phantom{olaola}= \int_\gamma \left( \int_C \frac{P_n^\mathsf{L} (z)\omega (t)P_m^\mathsf{R} (z)}{t-z}\, \frac{\operatorname d z}{2\pi \operatorname i} \right)\, \frac{\operatorname d t}{2\pi \operatorname i} \hspace{1.25cm} \text{(Fubini's theorem)}\\
&\phantom{olaola} = \int_\gamma P_n^\mathsf{L} (t)\omega (t) P_m^\mathsf{R} (t)\, \frac{\operatorname d t}{2\pi \operatorname i}\hspace{1.625cm} \text{(Cauchy's integral theorem)} 
\end{align*}
and so from~\eqref{eq:biorthogonality} we get the desired result.
\end{proof}

\subsubsection{Gauss--Borel factorization}

Let us give a nice interpretation of this biorthogonality.
First of all, we can see that the Stieltjes--Markov like transformation of the weight matrix~$ \omega$ is a generating function of the moments of $\omega$. In fact,
\begin{align*}
S_ \omega(z) = \int_{\gamma}\frac{ \omega(t)}{t-z}\frac{\operatorname d t}{2\pi \operatorname i}
 = \frac{1}{2\pi \operatorname i} \sum_{n=0}^\infty \frac{ \int_{\gamma} t^n \omega (t) \, \operatorname d t }{z^{n+1}}
 = \sum_{n=0}^\infty \frac{\omega_{n}}{z^{n+1}} , 
\end{align*}
for $|z| > r \coloneqq \max \big\{ |t|, t \in \gamma \big\} $.
Hence, $S_\omega$ is an analytic function on compact sets of $\mathbb C \setminus \big\{ z \in \mathbb C : |z| < r \big\}$.

We write down the biorthogonality conditions~\eqref{eq:ortogonalidadeS} that we have just derived in Theorem~\ref{pro:biorthogonalitystieltjes},
\begin{align*}
\int_C P_n^\mathsf{L} (z)S_\omega(z)P_m^\mathsf{R} (z)\, \frac{\operatorname d z}{2\pi \operatorname i}
 =
 \sum_{n=0}^\infty \int_C P_n^\mathsf{L} (z)\frac{\omega_{n}}{z^{n+1}} P_m^\mathsf{R} (z)\, \frac{\operatorname d z}{2\pi \operatorname i} 
\end{align*}
and by the Cauchy integral formula we get
\begin{align}
\label{eq:bior}
C_n^{-1} \delta_{n,m} = \int_C P_n^\mathsf{L} (z)S_\omega(z)P_m^\mathsf{R} (z)\, \frac{\operatorname d z}{2\pi \operatorname i}
 =
 \sum_{k=0}^\infty \frac{\Big( P_n^\mathsf{L} (z) \omega_{k} P_m^\mathsf{R} (z) \Big)^{(k)}\Big|_{z \to 0}}{k!}  .
\end{align}
Now, from the Leibniz rule for the derivatives, we know that
\begin{align*}
\frac1 {k!} \big( P_n^\mathsf{L} (z) \omega_{k} P_m^\mathsf{R} (z) \big)^{(k)}\Big|_{z \to 0} 
 =\sum_{j=0}^k \frac{ \big( P_n^\mathsf{L}\big)^{(j)} (0)}{j!} \omega_k \frac{ \big( P_m^\mathsf{R}\big)^{(k-j)} (0)}{(k-j)!} .
\end{align*}
With this identity we can reinterpret~\eqref{eq:bior} in matrix notation
\begin{align*}
\mathcal P^\mathsf L \, \pmb{\mathsf U} \, \mathcal P^\mathsf R = \operatorname{diag} \big\{ C_0^{-1} , C_1^{-1} , \ldots \big\}
\end{align*}
where
\begin{align*}
\pmb{\mathsf U} & = \begin{bmatrix}
\omega_0 & \cdots & \omega_n & \cdots \\
 \vdots & \ddots & \vdots \\
\omega_n & \cdots & \omega_{2n} & \cdots \\
\vdots & & \vdots & \ddots 
 \end{bmatrix} & \text{moment matrix}
 \\[.1cm]
\mathcal P^\mathsf L 
 & =
 \left.\begin{bmatrix}
 P_0^\mathsf L 
  & \\[.1cm]
 P_1^\mathsf L 
  & \big( P_1^\mathsf L\big)^\prime 
   & \\[.1cm]
 P_2^\mathsf L 
  & \big( P_2^\mathsf L\big)^\prime 
   & \frac 1{2!}\big( P_2^\mathsf L\big)^{\prime\prime} 
    \\[.1cm]
 \vdots & \vdots & \vdots & \ddots \\[.1cm]
 P_n^\mathsf L 
  & \big( P_n^\mathsf L\big)^\prime 
  & \frac 1{2!}\big( P_2^\mathsf L\big)^{\prime\prime} 
  & \cdots & & \frac 1{n!}\big( P_n^\mathsf L\big)^{(n)} 
  \\[.1cm]
 \vdots & & & & & \ddots
 \end{bmatrix} \right|_{z \to 0}
 &
\text{$
 P_n^\mathsf L 
 $
Taylor coefficients}
 \\[.1cm]
\mathcal P^\mathsf R 
 & =
 \left.\begin{bmatrix}
 P_0^\mathsf R 
 & \\[.1cm]
 P_1^\mathsf R 
 & \big( P_1^\mathsf R\big)^\prime 
 & \\[.1cm]
 P_2^\mathsf R 
  & \big( P_2^\mathsf R\big)^\prime 
  & \frac 1{2!}\big( P_2^\mathsf R\big)^{\prime\prime} 
  \\[.1cm]
 \vdots & \vdots & \vdots & \ddots \\[.1cm]
 P_n^\mathsf R 
  & \big( P_n^\mathsf R\big)^\prime 
  & \frac 1{2!}\big( P_2^\mathsf R\big)^{\prime\prime} 
  & \cdots & & \frac 1{n!}\big( P_n^\mathsf R\big)^{(n)} 
  \\[.1cm]
 \vdots & \vdots & \vdots & & & \ddots
 \end{bmatrix} \right|_{z \to 0}^\top
 &
\text{$ P_n^\mathsf R $ Taylor coefficients}
\end{align*}
As a conclusion:
We get the Gauss--Borel factorization of the moment matrix
\begin{align}
\label{eq:gaussborel}
\pmb{\mathsf U} =
\big( \mathcal P^\mathsf L\big)^{-1} \, \operatorname{diag} \big\{ C_0^{-1} , C_1^{-1} , \ldots \big\} \, \big( \mathcal P^\mathsf R \big)^{-1} .
\end{align}
We can see from~\eqref{eq:gaussborel} that the orthogonality relays on the Gauss--Borel factorization of the moment matrix. 

Remember that the necessary and sufficient conditions in order to assure that this representation takes place are exactly the ones we assume at the beginning in order to define the sequence of monic polynomials (left and right ones), i.e.~\eqref{eq:regular} takes place.

\subsection{Riemann--Hilbert problem} \label{subsec:14}

It can be seen~that
\begin{align*}
 {P}_n^{\mathsf L} (z) \, Q_0^{\mathsf L} (z) = \int_\gamma \big( {P}^{\mathsf L}_n (z) - {P}^{\mathsf L}_n (x) \big) \, 
 \frac{ \omega (x) }{z-x} \, \frac{\operatorname{d} x}{2\pi \operatorname i} +
 \int_\gamma {P}^{\mathsf L}_n (x) \, \frac{ \omega (x) }{z-x} \, \frac{\operatorname{d} x}{2\pi \operatorname i}
\end{align*}
and so, ${P}_{n-1}^{\mathsf L,(1)}$, defined by
\begin{align}
\label{eq:hpleft}
 {P}^{\mathsf L}_n (z) \, Q^{\mathsf L}_0 (z) - Q^{\mathsf L}_n (z) = 
 {P}_{n-1}^{\mathsf L,(1)} (z)
\end{align}
is for each $n \in \mathbb Z_+ \coloneqq \big\{ 1,2, \ldots \big\}$, a polynomial of degree at most $\displaystyle n-1 $, called the \emph{first kind associated polynomial} with respect to $\displaystyle \big\{ {P}_n^{\mathsf L} \, \big\}_{n \in \mathbb N} $ and $\displaystyle \omega $.
In fact,
\begin{align*}
{P}_{n-1}^{\mathsf L,(1)} (z)
  \coloneqq
\int_\gamma \big( {P}^{\mathsf L}_n (z) - {P}^{\mathsf L}_n (x) \big) \, \frac{ \omega (x) }{z-x} \,\frac{\operatorname{d} x}{2\pi \operatorname i} , && n \in \mathbb Z_+ . 
\end{align*}
We can matricially summarize the identities~\eqref{eq:ttrr}, \eqref{eq:ttrrql}, as
\begin{align*}
\left[
\begin{matrix}
 {P}^{\mathsf L}_{n+1} (z) & Q^{\mathsf L}_{n+1} (z) \\[.1cm]
C_{n} \, {P}_n^{\mathsf L} (z) & C_{n} \, Q_n^{\mathsf L} (z)
\end{matrix}
\right]
=
\left[
\begin{matrix}
z \operatorname{I} - \beta^{\mathsf L}_n & -C_{n}^{-1} \\[.1cm]
C_{n} & \pmb 0
\end{matrix}
\right]
\, \left[
\begin{matrix}
 {P}_{n}^{\mathsf L} (z) & Q_{n}^{\mathsf L} (z) \\[.1cm]
C_{n-1} \, {P}_{n-1}^{\mathsf L} (z) & C_{n-1} \, Q_{n-1}^{\mathsf L} (z)
\end{matrix}
\right] \, ;
\end{align*} 
and by~\eqref{eq:hpleft}
we also have that
\begin{align*}
\left[
\begin{matrix}
 {P}_{n}^{\mathsf L,(1)} (z) \\[.1cm]
C_{n} \, {P}_{n-1}^{\mathsf L,(1)} (z)
\end{matrix}
\right]
=
\left[
\begin{matrix}
z \operatorname{I} - \beta^{\mathsf L}_n & -C_{n}^{-1} \\[.1cm]
C_{n} & \pmb 0
\end{matrix}
\right]
\, \left[
\begin{matrix}
 {P}_{n-1}^{\mathsf L,(1)} (z) \\[.1cm]
C_{n-1} \, {P}_{n-2}^{\mathsf L,(1)} (z)
\end{matrix}
\right]
 . 
\end{align*}
Taking
\begin{align*}
 Y_n^{\mathsf L} (z) \coloneqq
\begin{bmatrix}
 {P}_{n}^{\mathsf L} (z) & Q_{n}^{\mathsf L} (z) \\[.1cm]
C_{n-1} \, {P}_{n-1}^{\mathsf L} (z) & C_{n-1} \, Q_{n-1}^{\mathsf L} (z)
\end{bmatrix}
 && \mbox{ and} &&
 T_n^{\mathsf L} \coloneqq
\begin{bmatrix}
z \operatorname{I} - \beta^{\mathsf L}_n & - C_{n}^{-1} \\[.1cm]
C_{n} & \pmb 0
\end{bmatrix} , 
\end{align*}
for all $ n \in \mathbb N$, we get
\begin{align}
\det \, Y_n^{\mathsf L} (z) = \det \, Y_{0}^{\mathsf L} (z) = 1 && \text{as 
}
 &&
\det \, T_n^{\mathsf L} = 1 , && n \in \mathbb N .
\end{align}
In the same way, we get from~\eqref{eq:rightttrr} and~\eqref{eq:ttrrqlr} that
\begin{align*}
Y_{n+1}^\mathsf R (z) = Y_{n}^\mathsf R (z) \, T_{n}^\mathsf R (z) , && n \in \mathbb N,
\end{align*}
where
\begin{align*}
Y_{n}^\mathsf R (z)
 \coloneqq \begin{bmatrix}
P^{\mathsf R}_{n} (z) & - P^{\mathsf R}_{n-1} (z) C_n \\[.1cm]
{Q}^{\mathsf R}_{n} (z) & - {Q}^{\mathsf R}_{n-1} (z) C_n
\end{bmatrix}
 && \text{and} &&
T_n^\mathsf R (z)
\coloneqq
\begin{bmatrix}
z \operatorname I - \beta_n^{\mathsf R} & - C_{n} \\[.1cm]
C_{n}^{-1} & \pmb{0}
\end{bmatrix}
\end{align*}
where $\beta_n^\mathsf R$ is defined in~\eqref{eq:rightcoeff}.
Here we also have
\begin{align}
\det \, Y_n^{\mathsf R} (z) = \det \, Y_{0}^{\mathsf R} (z) = 1
&&
 \text{as 
}
 &&
\det \, T_n^{\mathsf R} = 1 , && n \in \mathbb N .
\end{align}
As a conclusion:
We can establish that the matrix function $\displaystyle Y_n^{\mathsf L} $ (respectively $\displaystyle Y_n^{\mathsf R} $) is, for each $n \in \mathbb N$, the unique solution of the Riemann--Hilbert problem; which consists in determining a $2N \times 2 N$ matrix complex function
 $\displaystyle G_n^{\mathsf L} 
 $ (respectively $\displaystyle G_n^{\mathsf R}$) such that:
 \begin{description}
\item[(RH1)]
 $\displaystyle G_n^{\mathsf L} $ (respectively $\displaystyle G_n^{\mathsf R}$) is analytic in $\displaystyle \mathbb C \setminus \mathbb \gamma$;
\item[(RH2)] 
has the following asymptotic behavior near infinity,
\begin{align*}
 && G_n^{\mathsf L} (z) & = 
\big( \operatorname{I} + \operatorname{O} (z^{-1}) \big)
\begin{bmatrix}
\operatorname I \, z^n & \pmb 0 \\ \pmb 0 & \operatorname{I} \, z^{-n} 
\end{bmatrix}
 \\
\text{(respectively,} &&
 G_n^{\mathsf R} (z) & = 
 \begin{bmatrix}
\operatorname I z^n & \pmb{0} \\ 
 \pmb{0} & \operatorname I z^{-n} 
 \end{bmatrix}\big( \operatorname I + \operatorname{O} (z^{-1}) \big) \text{);}
\end{align*}
\item[(RH3)]
satisfies the jump condition
 $\displaystyle \big( G_n^{\mathsf L} (z) \big)_+ = \big( G_n^{\mathsf L} (z) \big)_- \,
\left[
\begin{matrix}
\operatorname I & \omega (x) \\ \pmb 0 & \operatorname{I}
\end{matrix}
\right] $ (respectively,
$\displaystyle \big( G^{\mathsf R}_n (z) \big)_+ =
 \begin{bmatrix}
 \operatorname I & \pmb{0} \\ 
 \omega (x) & \operatorname I 
 \end{bmatrix} \big( G^{\mathsf R}_n (z) \big)_- $), $\displaystyle x \in \gamma$.
\end{description}
Now, we will see that these two Riemann--Hilbert problems are related by similarity.

\begin{thm}
\label{teo:inversa}
Let $ Y_n^{\mathsf L} $ and $ Y_n^{\mathsf R} $ be, for each $n \in \mathbb N$, the unique solutions of the Rie\-mann--Hilbert problems just defined; then
\begin{align}
\label{eq:conexao}
\Big( Y_n^{\mathsf L}(z) \Big)^{-1} = 
\begin{bmatrix}
 \pmb 0 & \operatorname I \\ - \operatorname I & \pmb 0
\end{bmatrix}
Y_n^{\mathsf R} (z) 
\begin{bmatrix}
\pmb 0 & - \operatorname I \\ \operatorname I & \pmb 0
\end{bmatrix} , && n \in \mathbb N .
\end{align}
\end{thm}
This is an easy consequence of the Christoffel--Darboux identities, valid for all~$ n \in \mathbb N $,
\begin{align*}
(z-t) \sum_{k=0}^n P_k^{\mathsf R}(t) C_k P_k^{\mathsf L} (z) 
 & = P_n^{\mathsf R}(t) C_n P_{n+1}^{\mathsf L} (z)
 - P_{n+1}^{\mathsf R}(t) C_n P_{n}^{\mathsf L} (z) , \\
 (z-t) \sum_{k=0}^n Q_k^{\mathsf R}(t) C_k Q_k^{\mathsf L} (z) 
 & = Q_n^{\mathsf R}(t) C_n Q_{n+1}^{\mathsf L} (z)
 - Q_{n+1}^{\mathsf R}(t) C_n Q_{n}^{\mathsf L} (z) \\
 & \hspace{4.05cm} + S_\omega (z) - S_\omega (t), \\
 (z-t) \sum_{k=0}^n Q_k^{\mathsf R}(t) C_k P_k^{\mathsf L} (z) 
 & = Q_n^{\mathsf R}(t) C_n P_{n+1}^{\mathsf L} (z)
 - Q_{n+1}^{\mathsf R}(t) C_n P_{n}^{\mathsf L} (z) + \operatorname I
 \\
(z-t) \sum_{k=0}^n P_k^{\mathsf R}(t) C_k Q_k^{\mathsf L} (z) 
 & = P_n^{\mathsf R}(t) C_n Q_{n+1}^{\mathsf L} (z)
 - P_{n+1}^{\mathsf R}(t) C_n Q_{n}^{\mathsf L} (z) - \operatorname I .
\end{align*}
Taking $z=t$ we arrive to the identities
\begin{align*}
 P_n^{\mathsf R}(t) C_n P_{n+1}^{\mathsf L} (z) 
 - P_{n+1}^{\mathsf R}(t) C_n P_{n}^{\mathsf L} (z) & = \pmb 0 , \\
 Q_n^{\mathsf R}(t) C_n Q_{n+1}^{\mathsf L} (z) ,
 - Q_{n+1}^{\mathsf R}(t) C_n Q_{n}^{\mathsf L} (z) &= \pmb 0 , \\
Q_{n+1}^{\mathsf R}(t) C_n P_{n}^{\mathsf L} (z) - Q_n^{\mathsf R}(t) C_n P_{n+1}^{\mathsf L} (z)
 & = \operatorname I , \\
 P_n^{\mathsf R}(t) C_n Q_{n+1}^{\mathsf L} (z)
 - P_{n+1}^{\mathsf R}(t) C_n Q_{n}^{\mathsf L} (z) & = \operatorname I .
\end{align*}
All of this becomes 
\begin{align*}
\begin{bmatrix}
- Q_{n-1}^{\mathsf R} (z) C_{n-1} & - Q_{n}^{\mathsf R} (z) \\
 P_{n-1}^{\mathsf R} (z) C_{n-1} & P_{n}^{\mathsf R} (z) 
\end{bmatrix}
Y_n^{\mathsf L} (z) & = \operatorname I , & n \in \mathbb N ,
\end{align*}
and as
\begin{align*}
\begin{bmatrix}
- Q_{n-1}^{\mathsf R} (z) C_{n-1} & - Q_{n}^{\mathsf R} (z) \\
 P_{n-1}^{\mathsf R} (z) C_{n-1} & P_{n}^{\mathsf R} (z) 
\end{bmatrix}
 & = 
\begin{bmatrix}
\pmb 0 & \operatorname I \\ -\operatorname I & \pmb 0
\end{bmatrix}
Y_n^{\mathsf R} (z) 
\begin{bmatrix}
\pmb 0 & -\operatorname I \\ \operatorname I & \pmb 0
\end{bmatrix} , & n \in \mathbb N ,
\end{align*}
we get the desired result.

\section{Fundamental matrices 
from Riemann--Hilbert problems} \label{sec:2}

Here we consider weight matrices, $\omega$, satisfying a matrix Pearson type equation
\begin{align*}
\phi (z) \omega^\prime (z) = h^{\mathsf L} (z) \omega (z) + \omega (z) h^{\mathsf R} (z) ,
\end{align*}
where $\phi$ is a scalar polynomial of degree at most $2$, and $h^{\mathsf L} $,~$ h^{\mathsf R} $
are entire matrix functions.
If we take a weight function $\omega^{\mathsf L}$ such that
\begin{align}
\label{eq:pearsonjacobileft}
\phi (z) \big( \omega^{\mathsf L} \big)^\prime (z) = h^{\mathsf L} (z) \omega^{\mathsf L} (z) ,
\end{align}
then there exists a matrix function $\omega^\mathsf R (z)$ such that $\omega (z) = \omega^{\mathsf L}(z) \omega^{\mathsf R} (z)$ with
\begin{align}
\label{eq:pearsonjacobiright}
\phi (z) \big( \omega^{\mathsf R} \big)^\prime (z) = \omega^{\mathsf R} (z) h^{\mathsf R} (z) .
\end{align}
The reciprocal is also true.

For each factorization~$ \omega = \omega^{\mathsf L} \omega^{\mathsf R}$, we introduce the \emph{constant jump fundamental matrices } which will be instrumental in what follows
\begin{align}
\label{eq:zn}
 Z_n^{\mathsf L}(z) 
& \coloneqq Y^{\mathsf L}_n (z) 
\begin{bmatrix} 
 \omega^{\mathsf L} (z) & \pmb 0 \\ 
\pmb 0 & \big( \omega^{\mathsf R} (z)\big)^{-1} 
\end{bmatrix}
 , & \\
 \label{eq:zetan}
{Z}^{\mathsf R}_n (z)
& \coloneqq
\begin{bmatrix} 
 \omega^{\mathsf R} (z) & \pmb 0 \\ 
\pmb 0 & \big( \omega^{\mathsf L} (z)\big)^{-1} 
\end{bmatrix}
{Y}^{\mathsf R}_n (z) , & n \in \mathbb N
 .
\end{align}
Taking inverse on~\eqref{eq:zn} and applying~\eqref{eq:conexao} we see that 
$Z_n^{\mathsf R}$ given in~\eqref{eq:zetan} admits the representation
\begin{align}
\label{eq:znright}
Z_n^{\mathsf R} (z)
 & =
\begin{bmatrix}
\pmb 0 & - \operatorname I \\ \operatorname I & \pmb 0
\end{bmatrix}
\big(Z_n^{\mathsf L} (z)\big)^{-1}
\begin{bmatrix}
\pmb 0 & \operatorname I \\ - \operatorname I & \pmb 0
\end{bmatrix} , & n \in \mathbb N.
\end{align}
In parallel to the matrices $Z^{\mathsf L}_n(z)$ and $Z^{\mathsf R}_n(z)$, we introduce what we call \emph{structure matrices} given in terms of the \emph{left}, respectively \emph{right}, logarithmic derivatives by,
\begin{align}
\label{eq:Mn}
M^{\mathsf L}_n (z) & 
 \coloneqq \big( Z^{\mathsf L}_n \big)^{\prime} (z) \big(Z^{\mathsf L}_{n} (z)\big)^{-1},&
{M}^{\mathsf R}_n (z) & 
 \coloneqq \big ({Z}^{\mathsf R}_{n} (z)\big)^{-1} \big(Z^{\mathsf R}_n\big)^{\prime} (z) .
\end{align}
It is not difficult to see that
\begin{align}
\label{MnLR}
{M}^{\mathsf R}_{n}(z) &
 = -
\left[ \begin{matrix}
\pmb 0 & - \operatorname I \\ \operatorname I & \pmb 0
 \end{matrix} \right]
M_n^{\mathsf L} (z)
\left[ \begin{matrix}
\pmb 0 & \operatorname I \\ - \operatorname I & \pmb 0
 \end{matrix} \right] , &
n \in \mathbb N ,
\end{align}
as well as, the following properties hold~(cf.~\cite{BAFMM}):
\begin{enumerate}

\item
The transfer matrices satisfy
 \begin{align*}
 T^{\mathsf L}_n (z)Z_n^{\mathsf L}(z) & = Z^{\mathsf L}_{n+1}(z) ,& {Z}^{\mathsf R}_n(z){T}^{\mathsf R}_n (z) &=
 {Z}^{\mathsf R}_{n+1}(z), & n \in \mathbb N .
 \end{align*}

\item 
The zero curvature formulas holds,
\begin{align*}
\left[ \begin{matrix}
\operatorname I & \pmb 0 \\
\pmb 0 & \pmb 0 
 \end{matrix} \right]
 & 
= M^{\mathsf L}_{n+1} (z) T^{\mathsf L}_n(z) - T^{\mathsf L}_n (z) M^{\mathsf L}_{n} (z)
 , && n \in \mathbb N , 
 \\
\left[ \begin{matrix}
\operatorname I & \pmb 0 \\
\pmb 0 & \pmb 0
 \end{matrix} \right]
 & =
T^{\mathsf R}_n (z) \, M^{\mathsf R}_{n+1} (z) 
- M^{\mathsf R}_{n} (z) T^{\mathsf R}_n (z) , && n \in \mathbb N.
\end{align*}
\end{enumerate}
We see from~\eqref{eq:conexao},~\eqref{eq:znright} and~\eqref{MnLR} that we only need to consider the left side of the objects~$Y_n$, $Z_n$ and $M_n$.

\subsection{Hermite case}

In this case we consider $\phi (z) = 1$.
We underline that for a given regular weight matrix $ \omega (z)$ we will have many possible factorization $ \omega (z) = \omega^{\mathsf L} (z) \omega^{\mathsf R} (z)$. Indeed, if we define an equivalence relation $\big( \omega^{\mathsf L}, \omega^{\mathsf R} \big) \sim \big( \widetilde \omega^{\mathsf L},\widetilde \omega^{\mathsf R} \big)$ if, and only if, $ \omega^{\mathsf L} \omega^{\mathsf R} = \widetilde \omega^{\mathsf L}\widetilde \omega^{\mathsf R}$, then each weight matrix $ \omega $ can be though as a class of equivalence, and can be described by the orbit 
\begin{align*}
\big\{ \big( \omega^{\mathsf L} \varphi,\varphi^{-1} \omega^{\mathsf R} \big), \ \varphi \ \text{ is a nonsingular matrix of entire functions} \big\}.
\end{align*}
The constant jump fundamental matrices~$Z^{\mathsf L}_n(z) $ and~$Z^\mathsf R_n(z)$ are, for each $n \in \mathbb N$, characterized by the following properties: 
\begin{enumerate}

\item[{\rm i)}]
They are holomorphic on $\mathbb C \setminus \gamma$.

\item[{\rm ii)}]
We have the following asymptotic behaviors
\begin{align*}
Z^\mathsf L_n(z) &=\big( \operatorname I + \operatorname{O} (z^{-1}) \big)
\begin{bmatrix}
z^n \omega^{\mathsf L} (z) & \pmb{0} \\ 
\pmb{0} & z^{-n} ( \omega^{\mathsf R} (z))^{-1} 
\end{bmatrix}, \\
Z^\mathsf R_n(z)&
=\begin{bmatrix}
z^n \omega^{\mathsf R} (z) & \pmb 0 \\ 
\pmb 0 & ( \omega^{\mathsf L} (z))^{-1} z^{-n}
\end{bmatrix}\big( \operatorname I + \operatorname{O} (z^{-1}) \big), &&
\text{for $ z \to \infty$.}
\end{align*}

\item[{\rm iii)}]
They present the following \emph{constant jump condition} on $\gamma$
\begin{align*} 
\big( Z^{\mathsf L}_n (z) \big)_+ &= \big( Z^{\mathsf L}_n (z) \big)_- 
\begin{bmatrix}
\operatorname I & \operatorname I \\ 
\pmb 0 & \operatorname I 
\end{bmatrix} , &
\big( {Z}^{\mathsf R}_n (z) \big)_+ &=
\begin{bmatrix} 
\operatorname I & \pmb{0} \\ 
\operatorname I & \operatorname I 
\end{bmatrix} 
\big( {Z}^{\mathsf R}_n (z) \big)_- ,
\end{align*}
for all $z\in\gamma$ in the support on the weight matrix.
\end{enumerate}
In~\cite{BAFMM} we have proved that the structure matrices $M^{\mathsf L}_n (z) $ and $ M^{\mathsf R}_n (z) $, just defined,~cf.\eqref{eq:Mn}, are, for each $n \in \mathbb N$, matrices of entire functions in the complex~plane.

\subsection{Laguerre case}

Here we consider, $\phi (z) = z$, and $\omega$ a regular Laguerre~type weight matrix, i.e.
$
\omega =
\left[
\begin{smallmatrix}
\omega^{(1,1)} & \cdots & \omega^{(1,N)} \\
\vdots & \ddots & \vdots \\
\omega^{(N,1)} & \cdots & \omega^{(N,N)} 
\end{smallmatrix}
\right]
$,
with
\begin{align*}
\omega^{(j,k)}(z) = \sum_{m \in I_{j,k}} A_m(z) z^{\alpha_{m}} \log^{p_{m}} z 
 ,
&&
 z \in(0,+\infty),
\end{align*} 
where
$I_{j,k}$ denotes a finite set of indexes, $\operatorname{Re} \left( \alpha_{m}\right) > -1$, $p_{m} \in \mathbb{N} $ and $A_m(z)$ is H\"older continuous and bounded.

In~\cite{duran2004} different examples of Laguerre weights for the matrix orthogonal polynomials on the real line are studied. 

In order to state a Riemann--Hilbert problem for the Laguerre type weights we must add to the one presented in Section~\ref{subsec:14}, i.e. to (RH1)--(RH3), the condition
\begin{description}
\item[(RH4)]
$
Y^{\mathsf L}_n (z) 
= \left[\begin{matrix}
\operatorname{O} (1) & s^{\mathsf L}_{1}(z) \\[.1cm]
\operatorname{O} (1) & s^{\mathsf L}_{2}(z) 
\end{matrix} \right]$
and $
Y^{\mathsf R}_n (z) 
= 
\left[\begin{matrix}
\operatorname{O} (1) & \operatorname{O} (1) \\[.1cm]
s^{\mathsf R}_1(z) & s^{\mathsf R}_2(z) 
\end{matrix} \right]$, 
as $ z \to 0$,
with
$
\displaystyle 
\lim_{z \to 0} z s^{\mathsf L}_j(z) = \pmb 0
$,
$\displaystyle \lim_{z \to 0} z s^{\mathsf R}_j (z) = \pmb 0$,
$j=1,2$ and the~$ \operatorname{O} $ conditions are understood entrywise.
\end{description}
The solutions of~\eqref{eq:pearsonjacobileft} and~\eqref{eq:pearsonjacobiright} are of type
\begin{align*}
\omega^{\mathsf L}(z) = H^{\mathsf L}(z) z^{A^{\mathsf L}} \omega_0^{\mathsf L},
 &&
\omega^{\mathsf R}(z) =\omega_0^{\mathsf R} z^{A^{\mathsf R}} H^{\mathsf R}(z) 
\end{align*}
where 
$
H^{\mathsf L}$, $H^{\mathsf R}$ are entire and nonsingular matrix function such that $H^{\mathsf L}(0) = H^{\mathsf R}(0) = \operatorname I$, and~$\omega_0^{\mathsf L}$, $\omega_0^{\mathsf R}$ are a constant nonsingular matrix.

The constant jump fundamental matrices, introduced in~\eqref{eq:zn} and~\eqref{eq:zetan}, i.e.~$Z^{\mathsf L}_n$ and~$Z^\mathsf R_n$ satisfy, for each $n \in \mathbb N$, the fol\-low\-ing properties: 

\begin{enumerate}

\item[i)]
They are holomorphic on $\mathbb{C} \setminus \gamma$.

\item[ii)]
Present the fol\-low\-ing constant jump condition on $\gamma$
\begin{align*}
\big( Z^{\mathsf L}_n (z) \big)_+ &
= \big( Z^{\mathsf L}_n (z) \big)_- 
\left[\begin{matrix} 
(\omega_0^{\mathsf L})^{-1} e^{- 2 \pi \operatorname i A^{\mathsf L}} \omega_0^{\mathsf L} 
& (\omega_0^{\mathsf L})^{-1} e^{- 2 \pi \operatorname i A^{\mathsf L}} \omega_0^{\mathsf L}
 \\[.1cm]
\pmb 0 & 
\omega_0^{\mathsf R} e^{2 \pi \operatorname i A^{\mathsf R}} (\omega_0^{\mathsf R})^{-1}
\end{matrix} \right],\\
\big( Z^{\mathsf R}_n (z) \big)_+ &
=
\left[\begin{matrix} 
\omega_0^{\mathsf R} e^{-2 \pi \operatorname i A^{\mathsf R}} (\omega_0^{\mathsf R})^{-1} & \pmb{0}
 \\[.1cm]
\omega_0^{\mathsf R} e^{-2 \pi \operatorname i A^{\mathsf R}} (\omega_0^{\mathsf R})^{-1}& 
(\omega_0^{\mathsf L})^{-1}
e^{2 \pi \operatorname i A^{\mathsf L}} \omega_0^{\mathsf L}
\end{matrix} \right]
\big( Z^{\mathsf R}_n (z) \big)_- ,
\end{align*}
for all $z\in\gamma$,
where $A^\mathsf L = h^\mathsf L (0)$, $A^\mathsf R = h^\mathsf R (0)$.
\end{enumerate}
In~\cite{BFAM}, we discuss the holomorphic properties of the structure matrices introduced in~\eqref{eq:Mn}.
We could prove that, in the Laguerre case, the structure matrices~$M^{\mathsf L}_n(z) $ and $M^\mathsf R_n(z)$ are, for each $n \in \mathbb N$, meromorphic on~$\mathbb C$, with singularity located at $z=0$, which happens to be a removable singularity or a simple~pole.

\subsection{Jacobi case}

Here we follow~\cite{PAMS_AB_AF_MM} considering, $\phi (z) = z (1-z)$, and $\omega$ be a regular Jacobi~type weight matrix, i.e.
 $ \omega= 
\left[
\begin{smallmatrix} 
\omega^{(1,1)} & \cdots & \omega^{(1,N)} \\
\vdots & \ddots & \vdots \\
\omega^{(N,1)} & \cdots & \omega^{(N,N)} 
\end{smallmatrix}
\right]
$, supported on $\gamma$, with
\begin{align*}
\omega^{(j, k)}(z) = \sum_{m \in I_{j, k}} \varphi_{m}(z) z^{\alpha_{m}}(1-z)^{\beta_m} \log ^{p_{m}} (z)\log^{q_m}(1-z),
&& z \in \gamma ,
\end{align*}
where $I_{j, k}$ denotes a finite set of indexes, $\operatorname{Re} (\alpha_{m})$, $\operatorname{Re} (\beta_m)>-1$, $p_{m}$, $q_m \in \mathbb{N}$,
and $\varphi_{m}$ is
H\"older continuous, bounded and non-vanishing on~$\gamma$.
We assume that the determination of the logarithm and the powers are taken along~$\gamma$. We will request, in the development of the theory, that the functions~$\varphi_{m}$ have a holomorphic extension to the whole complex plane.

This case have been studied in~\cite{PAMS_AB_AF_MM} and includes the non scalar examples of Jacobi type weights given in the literature~\cite{nuevo,McCc,McAg,McAg2,AgPJ,SvA}.

In order to state a Riemann--Hilbert problem for the Jacobi type weights we must add to the one presented in Section~\ref{subsec:14}, i.e. to (RH1)--(RH3), the conditions
\begin{description}
\item[(RH4)]
$
Y^{\mathsf L}_n (z) 
= \left[ \begin{matrix}
\operatorname{O} (1) & s^{\mathsf L}_{1}(z) \\[.1cm]
\operatorname{O} (1) & s^{\mathsf L}_{2}(z) 
 \end{matrix} \right] $,
 $
Y^{\mathsf R}_n (z) 
= 
\left[ \begin{matrix}
\operatorname{O} (1) & \operatorname{O} (1) \\ 
s^{\mathsf R}_1(z) & s^{\mathsf R}_2(z) 
 \end{matrix} \right] $, 
as 
$z \to 0$,
with
$\displaystyle \lim_{z \to 0} z s^{\mathsf L}_j(z) = \pmb 0$
and
$\displaystyle \lim_{z \to 0} z s^{\mathsf R}_j (z) = \pmb 0$,
$j=1,2$.

\item[(RH5)]
$
Y_{n}^{\mathsf L}(z)=
\left[ \begin{matrix}
\operatorname{O}(1) & r_{1}^{\mathsf L}(z) \\[.1cm]
\operatorname{O}(1) & r_{2}^{\mathsf L}(z)
 \end{matrix} \right]$,
$
Y_{n}^{\mathsf R}(z)
=\left[\begin{matrix}
\operatorname{O}(1) & \operatorname{O}(1) \\
r_{1}^{\mathsf R}(z) & r_{2}^{\mathsf R}(z)
\end{matrix}\right]
$,
as 
$z \to 1$,
with
$\displaystyle \lim_{z \to 1} (1-z) r_{j}^{\mathsf L}(z)= \pmb 0$
and
$\displaystyle \lim_{z \to 1} (1-z) r_{j}^{\mathsf R}(z)= \pmb 0$,
$j=1,2$. 
The $s^{\mathsf L}_{i}$,~$s^{\mathsf R}_{i}$ (respectively, $r^{\mathsf L}_{i}$ and $r^{\mathsf R}_{i}$) could be replaced by $\operatorname{o}({1}/{z})$, as $z\to 0$ (respectively, $\operatorname{o}({1}/({1-z}))$, as $z\to 1$).
The $ \operatorname{O} $ and $\operatorname{o}$ conditions are understood entry-wise.
\end{description}
The solution of~\eqref{eq:pearsonjacobileft} and~\eqref{eq:pearsonjacobiright} will have possibly branch points
at $0$ and $1$, cf.~\cite{wasow}. This means that there exist constant matrices, $\mathsf C_j^\mathsf L$, $\mathsf C_j^\mathsf R$, with $j =0,1$, such that
\begin{align*}
( \omega^{\mathsf L} (z))_- = (\omega^{\mathsf L} (z))_+ \mathsf C_0^\mathsf L, &&
(\omega^{\mathsf R} (z))_- = \mathsf C_0^\mathsf R (\omega^{\mathsf R} (z))_+, && \text{in} & (0,1) ,
 \\
(\omega^{\mathsf L} (z))_- = (\omega^{\mathsf L} (z))_+ \mathsf C_1^\mathsf L, &&
(\omega^{\mathsf R} (z))_- = \mathsf C_1^\mathsf R (\omega^{\mathsf R} (z))_+, && \text{in} & (1,+\infty) .
\end{align*}
The constant jump fundamental matrices, introduced in~\eqref{eq:zn} and~\eqref{eq:zetan}, i.e.~$Z^{\mathsf L}_n$ and~$Z^\mathsf R_n$ satisfy, for each $n \in \mathbb N$, the fol\-low\-ing properties: 
\begin{enumerate}

\item
Are holomorphic on $\mathbb C \setminus [0,+\infty)$.
\item
Present the fol\-low\-ing \emph{constant jump condition} on $(0,1)$
\begin{align*}
\big( Z^{\mathsf L}_n (z) \big)_+ &
= \big( Z^{\mathsf L}_n (z) \big)_- 
\left[\begin{matrix} 
\mathsf C_0^{\mathsf L} & \mathsf C_0^{\mathsf L}
 \\[.1cm]
\pmb 0 & \operatorname I
\end{matrix} \right], 
 &&
\big( {Z}^{\mathsf R}_n (z) \big)_+ 
=
\left[\begin{matrix} 
\operatorname I & \pmb{0}
 \\[.1cm]
\mathsf C_0^{\mathsf R} & \mathsf C_0^{\mathsf R}
\end{matrix} \right]
\big( {Z}^{\mathsf R}_n (z) \big)_- .
\end{align*}

\item
Present the fol\-low\-ing \emph{constant jump condition} on $(1,+\infty)$
\begin{align*}
\big( Z^{\mathsf L}_n (z) \big)_+ 
& = 
\big( Z^{\mathsf L}_n (z) \big)_- 
\left[\begin{matrix} 
\mathsf C_1^{\mathsf L} & \pmb 0
 \\[.1cm]
\pmb 0 & \mathsf C_1^{\mathsf R}
\end{matrix} \right],
 &&
\big( {Z}^{\mathsf R}_n (z) \big)_+
=
\left[\begin{matrix} 
\mathsf C_1^{\mathsf R} & \pmb 0
 \\[.1cm]
\pmb 0 & \mathsf C_1^{\mathsf L}
\end{matrix} \right]
\big( {Z}^{\mathsf R}_n (z) \big)_- .
\end{align*}
\end{enumerate}
In~\cite{PAMS_AB_AF_MM}, we discuss the holomorphic properties of the structure matrices introduced in~\eqref{eq:Mn}.
We could prove that, in the Jacobi case, the structure matrices~$M^{\mathsf L}_n(z) $ and $M^\mathsf R_n(z)$ are, for each $n \in \mathbb N$, meromorphic on~$\mathbb C$, with singularities located at $z=0$ and $z=1$, which happens to be a removable singularity or a simple~pole.

\subsection{First order equations}

The analytic properties of the matrix, $M_n^\mathsf L$, just studied, represent an open door to the differential properties of $\big\{ Z_n^{\mathsf L} \big\}_{n \in \mathbb N}$, defined in~\eqref{eq:zn}, as well as for $\big\{ Y_n^{\mathsf L} \big\}_{n \in \mathbb N}$ (the same can be done in the right case). In fact,
from the analytic properties of $M_n^\mathsf L$ we~get~that
\begin{align}
\label{eq:primeiraordemZnL}
\phi (z) \, \big( Z_n^\mathsf L (z) \big)^\prime = \widetilde M_n^\mathsf L (z) \, Z_n^\mathsf L (z) , && n \in \mathbb N .
\end{align}
where $ \big\{ \widetilde M_n^\mathsf L \big\}_{n \in \mathbb N} $ is a sequence of entire functions defined by
\begin{align}
\label{eq:wide}
\widetilde M_n^\mathsf L (z) \coloneqq \phi (z) \, M_n^\mathsf L (z) , && n \in \mathbb N ,
\end{align} 
From~\eqref{eq:znright} and taking into account
\begin{align*}
\Big( \big( Z_n^\mathsf L (z) \big)^{-1}\Big)^\prime
 = - \big( Z_n^\mathsf L (z) \big)^{-1} \big( Z_n^\mathsf L (z) \big)^{\prime} \big( Z_n^\mathsf L (z) \big)^{-1} ,
\end{align*}
we arrive to
\begin{align}\label{eq:primeiraordemZnR}
\phi (z) \, \big( Z_n^\mathsf R (z) \big)^\prime = Z_n^\mathsf R (z) \, \widetilde M_n^\mathsf R (z) , && n \in \mathbb N ,
\end{align}
where by~\eqref{MnLR}
\begin{align*}
\widetilde M_n^\mathsf R (z) 
 & = 
-
\left[ \begin{matrix}
\pmb 0 & - \operatorname I \\ \operatorname I & \pmb 0
 \end{matrix} \right]
 \widetilde M_n^\mathsf L (z) 
 \left[ \begin{matrix}
\pmb 0 & \operatorname I \\ - \operatorname I & \pmb 0
 \end{matrix} \right] ,
&
n \in \mathbb N .
\end{align*}
Hence, the entries of the matrices
\begin{align*}
\widetilde M_n^\mathsf L (z) =
\begin{bmatrix}
\mathsf L_n^{1,1} & \mathsf L_n^{1,2} \\ \mathsf L_n^{2,1} & \mathsf L_n^{2,2}
\end{bmatrix}
 && \text{and} &&
\widetilde M_n^\mathsf R (z) =
\begin{bmatrix}
\mathsf R_n^{1,1} & \mathsf R_n^{1,2} \\ \mathsf R_n^{2,1} & \mathsf R_n^{2,2}
\end{bmatrix} ,
\end{align*}
are related by
\begin{align*}
 \mathsf R_n^{2,2} = - \mathsf L_n^{1,1} , &&
 \mathsf R_n^{2,1} = \mathsf L_n^{1,2} , &&
 \mathsf R_n^{1,2} = \mathsf L_n^{2,1} , &&
 \mathsf R_n^{1,1} = - \mathsf L_n^{2,2} , && n \in \mathbb N .
\end{align*}
Now, from~\eqref{eq:primeiraordemZnL} and~\eqref{eq:primeiraordemZnR} we get the first order structure relations for the sequences $\big\{ P_n^\mathsf L \big\}_{n \in \mathbb N}$, $\big\{ Q_n^\mathsf L \big\}_{n \in \mathbb N}$, $\big\{ P_n^\mathsf R \big\}_{n \in \mathbb N}$, and $\big\{ Q_n^\mathsf R \big\}_{n \in \mathbb N}$
\begin{align*}
\phi (z) \big(P^\mathsf L_n (z) \big)^{\prime}+ P^\mathsf L_n (z) h^\mathsf L(z)
 & = \mathsf L^{1,1}_n (z) P^\mathsf L_n (z) - \mathsf L^{1,2}_n (z) C_{n-1} P^\mathsf L_{n-1} (z) , 
 \\ 
\phi (z) \big(Q^\mathsf L_n (z) \big)^{\prime} - Q^\mathsf L_n (z) h^\mathsf R(z) 
 & = \mathsf L^{1,1}_n (z) Q^\mathsf L_n (z) - \mathsf L^{1,2}_n (z) C_{n-1} Q^\mathsf L_{n-1} (z) , \\
\phi (z) \big(P^\mathsf R_n (z) \big)^{\prime}+ h^\mathsf R(z) P^\mathsf R_n (z)
 & = - P^\mathsf R_n (z) \mathsf L^{2,2}_n (z) - P^\mathsf R_{n-1} (z) C_{n-1} \mathsf L^{1,2}_n (z) ,
  \\
\phi (z) \big(Q^\mathsf R_n (z) \big)^{\prime} - h^\mathsf L(z)Q^\mathsf R_n (z) 
 & = - Q^\mathsf R_n (z) \mathsf L^{2,2}_n (z) - Q^\mathsf R_{n-1} (z) C_{n-1} \mathsf L^{1,2}_n(z) .
\end{align*}
Next we give the representation of the matrix $M_n^\mathsf L$ when we are in one of the cases we have just studied, with
$h^{\mathsf L} (z)=A^{\mathsf L} z+B^{\mathsf L}$
and
$h^{\mathsf R}(z)=A^{\mathsf R} z+B^{\mathsf R}$.

For example in~\cite{BAFMM} we get for the matrix $M_n^\mathsf L$ in the Hermite case the representation
$M^\mathsf L_n (z) = {\mathcal A}^\mathsf L z + \mathcal K_n^\mathsf L$, with
\begin{align*} 
\mathcal A^\mathsf L & 
 = 
\left[ \begin{smallmatrix}
A^{\mathsf L} & 0_N \\
0_N & -A^{\mathsf R} \end{smallmatrix} \right]
, &
\mathcal K_n^\mathsf L 
& 
 =
\left[ \begin{smallmatrix}
B^{\mathsf L} + \big[ p_{\mathsf L , n}^1 , A^{\mathsf L} \big] 
 & C_n^{-1} A^{\mathsf R} + A^{\mathsf L} C_n^{-1} \\
- C_{n-1} A^{\mathsf L} - A^{\mathsf R} C_{n-1} 
 & -B^{\mathsf R} -\big[ q_{\mathsf L , n-1}^1 , A^{\mathsf R} \big]
\end{smallmatrix} \right] , && n \in \mathbb N.
\end{align*}
In the Laguerre case we get following~\cite{BFAM} that the matrix $\widetilde M_n^\mathsf L$ defined in~\eqref{eq:wide} is given~by
\begin{align*}
\widetilde M^{\mathsf L}_n (z) & =
\left[ \begin{matrix} 
A^{\mathsf L} z +[ p^1_{\mathsf L,n}, A^{\mathsf L} ] + n I_N + B^{\mathsf L} & A^{\mathsf L} C_n^{-1} + C_n^{-1} A^{\mathsf R} \\[.1cm]
-C_{n-1} A^{\mathsf L} -A^{\mathsf R} C_{n-1} & -A^{\mathsf R} z + [ p^1_{\mathsf R,n} ,A^{\mathsf R} ] - n I_N- B^{\mathsf R} 
 \end{matrix} \right] , && n \in \mathbb N.
\end{align*}
For the Jacobi case we get in~\cite{PAMS_AB_AF_MM}, for all $ n \in \mathbb N$, that
the matrix $ \widetilde{M}^{\mathsf L}_n$ defined by~\eqref{eq:wide} is
\begin{align*}
 \widetilde{M}^{\mathsf L}_n (z) & = 
\resizebox{.855\hsize}{!}{$
 \left[ \begin{matrix} 
\left(A^{\mathsf L}-nI_N\right) z +[ p^1_{\mathsf L,n}, A^{\mathsf L} ] + p^1_{\mathsf L,n}+ n I_N + B^{\mathsf L}
 &
A^{\mathsf L} C_n^{-1} + C_n^{-1} A^{\mathsf R}-(2n+1)C_{n}^{-1}
 \\[.15cm]
-C_{n-1} A^{\mathsf L} -A^{\mathsf R} C_{n-1}+(2n-1)C_{n-1}
&
\left(nI_N-A^{\mathsf R}\right) z + [ p^1_{\mathsf R,n} ,A^{\mathsf R} ]- p^1_{\mathsf R,n} - n I_N- B^{\mathsf R}
 \end{matrix} \right] 
 $} .
 \end{align*}

\section{Second order differential relations} \label{sec:3}

Taking derivative on~\eqref{eq:primeiraordemZnL} we arrive to
\begin{multline*}
\phi^\prime (z) \big( Z_n^\mathsf L (z) \big)^\prime \big( Z_n^\mathsf L (z) \big)^{-1}
+ \phi (z) \big( Z_n^\mathsf L (z)\big)^{\prime\prime} \big( Z_n^\mathsf L (z)\big)^{-1}
 \\ - \phi (z) \big( Z_n^\mathsf L (z) \big)^{\prime} \big( Z_n^\mathsf L (z) \big)^{-1}
\big( Z_n^\mathsf L (z) \big)^{\prime} \big( Z_n^\mathsf L (z)\big)^{-1}
 = \big( \widetilde M_n^\mathsf L (z) \big)^\prime ,
\end{multline*}
and using again~\eqref{eq:primeiraordemZnL} we get
\begin{align*}
\phi (z) \big( Z_n^\mathsf L (z) \big)^{\prime\prime} 
 = \left\{ \big( \widetilde M_n^\mathsf L (z) \big)^\prime - \frac{\phi^\prime (z) }{\phi (z)} \widetilde M_n^\mathsf L (z)
 + \frac{\big( \widetilde M_n^\mathsf L (z)\big)^2}{\phi (z)} \right\} \, Z_n^\mathsf L (z) ,
\end{align*}
From~\eqref{eq:pearsonjacobileft} and~\eqref{eq:pearsonjacobiright} we have
\begin{align*}
\phi(z){\left( \omega^{\mathsf L} (z) \right)}^{\prime\prime}{\left( \omega^{\mathsf L} (z) \right)}^{-1}
 & = {\big(h^{\mathsf L} (z) \big)}^\prime -\frac{\phi^\prime (z)}{ \phi(z)} h^{\mathsf L} (z)
\frac{{\big(h^{\mathsf L} (z) \big)}^2}{\phi(z)} , \\
\phi (z){\left( \big( \omega^{\mathsf R} (z) \big)^{-1}\right)}^{\prime\prime} \omega^{\mathsf R} (z)
 & =
\frac{\big({h^{\mathsf R} (z) \big)}^2}{\phi(z)}+\frac{\phi^\prime (z)}{\phi(z)}h^{\mathsf R} (z)-{\big(h^{\mathsf R} (z) \big)}^\prime .
\end{align*}
Now, since
\begin{multline*}
\phi(z)\left(Z_{n}^{\mathsf L}\right)^{\prime \prime}{\left(Z_{n}^{\mathsf L}\right)}^{-1}
=\phi (z){(Y_n^{\mathsf L})}^{\prime\prime}Y_n^{\mathsf L}+{\left(Y_n^{\mathsf L}\right)}^\prime
\left[\begin{matrix}
2h^{\mathsf L}& 0_N \\
0_N & -2h^{\mathsf R}
\end{matrix}
\right]
{(Y_n^{\mathsf L})}^{-1} \\
+Y_n^{\mathsf L}\left[\begin{matrix}
{\phi(z)\left( \omega^{\mathsf L}\right)}^{\prime\prime}{\left( \omega^{\mathsf L}\right)}^{-1} & \pmb 0 \\
\pmb 0 & { \phi(z)\left(( \omega^{\mathsf R})^{-1}\right)}^{\prime\prime} \omega^{\mathsf R} 
\end{matrix}\right]
{(Y_n^{\mathsf L})}^{-1} ,
\end{multline*}
we finally get
\begin{multline}
\label{eq:secondorderL} 
\phi(z) \big(Y^\mathsf L_n\big)^{\prime\prime} 
+ \big(Y^\mathsf L_n\big)^{\prime} 
\left[ \begin{matrix} 
 2 h^{\mathsf L}+ \phi^\prime(z) \operatorname I & \pmb 0 \\ 
 \pmb 0 & - 2 h^{\mathsf R} + \phi^\prime( z) \operatorname I
 \end{matrix} \right] 
 \\
 + Y_n^\mathsf L (z)
 \left[ \begin{matrix}
 \mathcal N (h^{\mathsf L}) & \pmb 0 \\ 
 \pmb 0 &\mathcal N (- h^{\mathsf R}) 
 \end{matrix} \right] 
 = \mathcal N ( \widetilde{M}^\mathsf L_n )Y^\mathsf L_n ,
 \end{multline}
 where $\displaystyle 
\mathcal N (F(z)) = F^{\prime}(z)+\frac{F^2(z)}{\phi(z)}$.

The same procedure leads us to
\begin{multline}
\label{eq:secondorderR}
\phi(z) \big(Y^\mathsf R_n\big)^{\prime\prime} 
+
\left[ \begin{matrix} 
2 h^{\mathsf R} + \phi^\prime (z) \operatorname I & \pmb 0 \\ 
\pmb 0 & - 2 h^{\mathsf L} + \phi^\prime (z) \operatorname I 
 \end{matrix} \right] 
 \big(Y^\mathsf R_n\big)^{\prime} 
 \\ + 
\left[ \begin{matrix} 
 \mathcal N (h^{\mathsf R}) & \pmb 0 \\ 
 \pmb 0 & \mathcal N (-h^{\mathsf L}) 
 \end{matrix} \right] 
Y_n{^\mathsf R} (z)
 = 
Y^\mathsf R_n \mathcal N (\widetilde{M}^\mathsf R_n )
 .
 \end{multline}
We can see that~\eqref{eq:secondorderL} and\eqref{eq:secondorderR} enclose second order differential relations for $P_n^\mathsf L$, $P_n^\mathsf R$, $Q_n^\mathsf L$ and $Q_n^\mathsf R$ that we have explicitly determine in~\cite{BAFMM}, \cite{BFAM} and~\cite{PAMS_AB_AF_MM}, for the Hermite, Laguerre and Jacobi cases, respectively.

We could recover the scalar classical cases of Hermite, Laguerre and Jacobi form the study we have presented for the matrix~cases.

\subsection{Hermite case}

For example for the scalar Hermite case we get that
\begin{align*}
H_n^{\prime\prime} (z) - 2 z H_n^{\prime} (z) & = - 4 \gamma_n H_n (z) , \\ 
Q_n^{\prime\prime} (z) + 2 z Q_n^{\prime} (z) & = - ( 4 \gamma_n + 2 ) Q_n (z) , && n \in \mathbb N ,
\end{align*}
where, $\big\{ H_n \big\}_{n \in \mathbb N}$ is the sequence of monic orthogonal polynomials orthogonal with respect to the weight $w (x) = e^{-x^2}$ in $\mathbb R$, satisfying a three term recurrence relation
\begin{align*}
x H_n (x) = H_{n+1} (x) + \gamma_n H_{n-1} (x) , && n \in \mathbb N ,
\end{align*}
with $H_{-1} (x) = 0$, $H_0 (x) = 1$, $\gamma_{n} = \frac n 2$, and $ \big\{ Q_n \big\}_{n \in \mathbb N}$ is the sequence of second kind functions associated with $\big\{ H_n \big\}_{n \in \mathbb N}$ and $w$.

\subsection{Berezanskii Laguerre case}

In the scalar monic Laguerre polynomials, $\big\{ L_n^\alpha \big\}_{n \in \mathbb N}$, orthogonal with respect to the weight $w^\alpha (x) = x^\alpha e^{-x}$ in
 \linebreak
$(0,+\infty)$, with $\alpha \in ( -1, +\infty ) 
$, and second kind functions, $\big\{ Q_n^\alpha \big\}_{n \in \mathbb N}$, we got in~\cite{BFAM} that
\begin{align*}
z \big( L_n^\alpha \big)^{\prime\prime} (z) - ( z - \alpha -1) \big( L_n^\alpha \big)^{\prime} (z) &= - n \, L_n^\alpha (z) , \\ 
z \big( Q_n^\alpha \big)^{\prime\prime} (z) + ( z - \alpha +1 ) \big( Q_n^\alpha \big)^{\prime} (z) &= - ( n + 1 ) \, Q_n^\alpha (z) , && n \in \mathbb N.
\end{align*}
Now, considering the Berezanskii example, of monic matrix orthogonal polynomials, $\big\{ \mathbb P_n \big\}_{n \in \mathbb N}$, given in~\eqref{eq:diagonalpol}, with $\omega^1 (x) = w^\alpha (x)$ and $\omega^2 (x) = w^\beta (x)$, we get that
\begin{align*}
z \mathbb P_n^{\prime\prime} (z) - 
\mathbb P_n^{\prime} (z) \Psi_1 (z) &= - n \, \mathbb P_n (z) , \\ 
z \mathbb Q_n^{\prime\prime} (z) +
\mathbb Q_n^{\prime} (z) \Psi_2 (z) &= - (n+1) \, \mathbb Q_n (z) , && n \in \mathbb N ,
\end{align*}
where
\begin{align*}
\Psi_1 (z)
 & =
\frac 1 2
\begin{bmatrix}
 - {\alpha} - {\beta} + 2(z-1) & {\beta-\alpha} \\[.1cm]
 {\beta-\alpha} & - {\alpha} - {\beta} + 2(z-1)
\end{bmatrix}
 , 
 \\
\Psi_2 (z)
 & =
\frac 1 2
\begin{bmatrix}
 - {\alpha} - {\beta} + 2(z+1) & {\beta-\alpha} \\[.1cm]
 {\beta-\alpha} & - {\alpha} - {\beta} + 2(z+1)
\end{bmatrix}
 .
\end{align*}

\subsection{Berezanskii Jacobi case}

Let us consider the weight $W^{\alpha, \beta}(z)=z^\alpha (1-z)^\beta$, in $[-1,1]$, with
 $\alpha$, $\beta$ scalars in~$(-1, \infty)$. Then, the scalar second order equation for $\big\{p_n\big\}_{n\in\mathbb N}$ and $\big\{q_n\big\}_{n\in\mathbb N}$ (cf.~for example~\cite{Szego}) is given by
\begin{align*}
 & z(1-z) p_n^{\prime\prime} (z) + \big(1+\alpha - (\alpha+\beta +2)z\big) p_n^{\prime} (z) + n(\alpha+\beta+n+1) p_n (z) =0 , \\ 
 & z(1-z) q_n^{\prime\prime} (z) + \big(1-\alpha + (\alpha+\beta -2)z\big) q_n^{\prime} (z) +(n+1)(\alpha+\beta+n) q_n (z) =0 .
\end{align*}
Now, we can construct the Berezanskii monic matrix orthogonal polynomials, $\big\{ \mathbb P_n \big\}_{n \in \mathbb N}$, given in~\eqref{eq:diagonalpol}, with $\omega^1 (x) = W^{\alpha,\beta} (x)$ and $\omega^2 (x) = W^{\beta,\alpha} (x)$, we get that
\begin{align*}
z \mathbb P_n^{\prime\prime} (z) + 
\mathbb P_n^{\prime} (z) \Psi_1 (z) + n(\alpha+\beta+n+1) \mathbb P_n (z) &= \pmb 0 , \\ 
z \mathbb Q_n^{\prime\prime} (z) +
\mathbb Q_n^{\prime} (z) \Psi_2 (z) + (n+1)(\alpha+\beta+n) \mathbb Q_n (z) &= \pmb 0 , && n \in \mathbb N ,
\end{align*}
where
\begin{align*}
\Psi_1 (z)
 & =
\frac 1 2
\begin{bmatrix}
 -{\alpha}-{\beta}+ 2(z -1) & {\beta-\alpha} \\[.1cm]
 {\beta-\alpha} & -{\alpha}- {\beta} + 2(z -1)
\end{bmatrix}
 , \\
\Psi_2 (z)
 & =
\frac 1 2
\begin{bmatrix}
 (2 z-1) (\alpha+\beta-2) & {\beta-\alpha} \\[.1cm]
 {\beta-\alpha} & (2 z-1) (\alpha+\beta-2)
\end{bmatrix}
 .
\end{align*}

\section{Discrete Painlev\'e type matrix equations} \label{sec:4}

We now discuss the case 
of a weight $\omega$
satisfying a generalized matrix Pearson equation
\begin{align*}
 \omega^{\prime} (z) = (\lambda + \mu z + \nu z^2 ) \omega (z)
 , && \text{with} && \mu,\nu\in\mathbb C^{N\times N} .
\end{align*}
The structure matrix, cf.~\eqref{eq:Mn}, is given by
$M_n (z) = M_n^0 z^2 + M_n^1 z + M_n^2$~with
\begin{align*}
M_n^0 &= 
\begin{bmatrix} 
\nu & \pmb 0 \\ 
\pmb 0 & \pmb 0
\end{bmatrix}, 
\phantom{olaolaola} 
M_n^1 
 = 
\begin{bmatrix} 
\mu - \big[ \nu , p_n^1 \big] & \nu C_n^{-1} \\[.05cm] 
- C_{n-1} \nu & \pmb{0} 
\end{bmatrix}, 
\end{align*}
\begin{align*}
M_n^2 &=
\resizebox{.855\hsize}{!}{$
\begin{bmatrix} 
\lambda - \big[ \beta , p_n^1 \big] - \big[ \nu , p_n^2 \big] + \nu \big( p_n^1 \big)^2 - p_n^1 \nu \, p_n^1 + \nu C_n^{-1} C_{n-1} & \big( \mu - \big[ \nu , p_n^1 \big] + \gamma \beta_n \big) C_n^{-1} \\[.05cm] 
- C_{n-1} \, \big( \mu + p_{n-1}^1 \nu - \nu p_{n}^1 \big) & 
- C_{n-1} \nu \, C_n^{-1}
\end{bmatrix} $}.
\end{align*}
Then in~\cite{BAFMM}, we prove that the three term recursion coefficients $\gamma_n$ can be expressed directly in terms of the recursion coefficients $\beta_n$, for all $n \in \mathbb N$,
\begin{align*}
\gamma_{n+1}=-(n+1) \Big( \beta + \Big[ \gamma , \sum\limits_{k=0}^{n-1} \beta_k \Big] + \gamma (\beta_n + \beta_{n+1} ) \Big)^{-1}.
\end{align*}
The coefficients $\beta_n$ fulfill, for all $n \in \mathbb N$, the following non-Abelian alt-dPI,
 \begin{multline*}
\lambda + \nu \big( \gamma_n + \gamma_{n+1} + \beta_n^2 \big) - \mu \beta_n + \Big[ \beta , \sum\limits_{k=0}^{n-1}\beta_k \Big] \big(I_N+ \beta_n\big) \\+ \big[ \nu , \sum\limits_{m=1}^{n-1}\gamma_m-\sum\limits_{
 0\leq k<m\leq n-1}\beta_m\beta_k \big] 
+ \Big[ \nu , \sum\limits_{k=0}^{n-1}\beta_k \Big] \sum\limits_{k=0}^{n-1}\beta_k = \pmb 0 .
\end{multline*}
Now, we consider a Laguerre type weight matrix $W $ such that
\begin{align*}
z W^{\prime}(z) = ( h_0 + h_1 z + h_2 z^2) W (z) .
\end{align*}
Then, the entries of the structure matrix, $\widetilde M_n$, cf.~\eqref{eq:wide}, are given by
\begin{align*}
 \widetilde{M}_n^{11} & = 
(h_0 + h_1 z + h_2 z^2) + h_1 q_{\mathsf R,n-1}^{1} + p_{\mathsf L,n}^{1} h_1 \\
& \phantom{olaola} + z
(h_2 q_{\mathsf R,n-1}^{1} + p_{\mathsf L,n}^{1} h_2)
+ h_2 q_{\mathsf R,n-1}^{2}
+ p_{\mathsf L,n}^{2} h_2
+ p_{\mathsf L,n}^{1} h_2 q_{\mathsf R,n-1}^{1}
+n \operatorname I
 , \\
 \widetilde{M}_n^{12} & = (h_1 + h_2 z + h_2 q_{\mathsf R,n}^{1} + p_{\mathsf L,n}^{1} h_2)
C_n^{-1} ,
\end{align*}
\begin{align*}
\widetilde{M}_n^{21} & = -C_{n-1} (h_1 + h_2 z +
h_2 q_{\mathsf R,n-1}^{1} + p_{\mathsf L,n-1}^{1} C_{\mathsf L}) ,
\\
\widetilde{M}_n^{22} & = -C_{n-1} h_2 C_n^{-1} 
 -n \operatorname I .
\end{align*}
From these we proved in~\cite{BFAM} that the following system of matrix equations is a noncommutative version of an instance of discrete Painlev\'e~IV equation:
\begin{align*}
& (2n+1) \operatorname I + h_0 
+ h_2 (\gamma_{n+1} + \gamma_{n} ))
+(h_2 \beta_{n} + h_1 ) \beta_{n}
 \\
& \phantom{olaola} = 
 \Big[ \sum_{k=0}^{n-1} \beta_k , h_2 \Big] \sum_{k=0}^{n} \beta_k
- \Big[\sum_{i,j=0}^{n-1} \beta_i \beta_j - \sum_{k=0}^{n-1}\gamma_k
, h_2 \Big] - \Big[ \sum_{k=0}^{n-1} \beta_k , h_1 \Big], \\
 & \beta_n -\gamma_{n} \big( h_2 (\beta_n + \beta_{n-1} ) + h_1 \big) 
+ \big( h_2 (\beta_n + \beta_{n+1} ) + h_1 \big) \gamma_{n+1}
 \\
 &\phantom{olaola}
 = -\gamma_{n} \Big[ \sum_{k=0}^{n-1} \beta_k , h_2 \Big] +
\Big[-\sum_{k=0}^{n-1} \beta_k , h_2 \Big] \gamma_{n+1} .
\end{align*}
We can also consider, the weight matrix $W (z) $
such that
\begin{align*}
z(1-z) W^{\prime} (z) &= ( h_0 +
h_1 z + h_2 z^2) W (z),
\end{align*}
which is a Jacobi type weight matrix.
It was proved in~\cite{PAMS_AB_AF_MM} that the entries of the structure matrix, $\widetilde M_n$, cf.~\eqref{eq:wide}, are given by
\begin{align*}
& \widetilde{M}_n^{11} = 
(h_0 + h_1 z + h_2 z^2) + h_1
q_{\mathsf R,n-1}^{1} + p_{\mathsf L,n}^{1} h_1 + z(h_2 q_{\mathsf R,n-1}^{1}+ p_{\mathsf L,n}^{1} h_2) \\
& \phantom{olaolaolaolaola}
+ h_2 q_{\mathsf R,n-1}^{2}
+ p_{\mathsf L,n}^{2} h_2
+ p_{\mathsf L,n}^{1} h_2 q_{\mathsf R,n-1}^{1}+n \operatorname I -z n \operatorname I +p_{\mathsf L,n}^1 , 
\\
& \widetilde{M}_n^{12} = (h_1 + h_2 z + h_2 q_{\mathsf R,n}^{1} + p_{\mathsf L,n}^{1} h_2 )
C_n^{-1} -(2n+1)C_{n}^{-1} ,
\\
& \widetilde{M}_n^{21} = -C_{n-1} (h_1 + h_2 z +
h_2 q_{\mathsf R,n-1}^{1} + p_{\mathsf L,n-1}^{1} h_2) +(2n-1)C_{n-1} ,
\\
 & \widetilde{M}_n^{22} = -C_{n-1} h_2 C_n^{-1} 
-n \operatorname I +z n \operatorname I -p_{\mathsf R,{n}}^1 .
\end{align*}
We plug this information in the equations of zero curvature presented in Section~\ref{sec:2} and get
\begin{align*}
& \hspace{-.225cm} (2n+1) \operatorname I + h_0
+ h_2 (\gamma_{n+1} + \gamma_{n} )
+\left(h_2 \beta_{n} + h_1 -(2n+1) \operatorname I \right) \beta_{n} +\sum_{k=0}^{n-1} \beta_k \\
& \hspace{-.225cm} \nonumber
 +C_{n}^{-1}\sum_{k=0}^{n} \beta_k C_{n} = 
 \Big[ \sum_{k=0}^{n-1} \beta_k , h_2 \Big] \sum_{k=0}^{n} \beta_k 
- \Big[\sum_{i,j=0}^{n-1} \beta_i \beta_j - \sum_{k=0}^{n-1}\gamma_k
, h_2 \Big] - \Big[ \sum_{k=0}^{n-1} \beta_k, h_1 \Big],
 \\
 & \hspace{-.225cm}
\beta_n -(\beta_n )^2-\gamma_{n} \big( h_2 (\beta_n + \beta_{n-1} ) + h_1 -(2n-1) \operatorname I \big) + \big( h_2 (\beta_n + \beta_{n+1}) + h_1 \\
 &\hspace{-.225cm} \nonumber
-(2n+3) \operatorname I \big) \gamma_{n+1}
= \gamma_{n} \Big[ \sum_{k=0}^{n-2} \beta_k , h_2 \Big] -
\Big[\sum_{k=0}^{n-1} \beta_k , h_2 \Big] \gamma_{n+1} -\Big[\sum_{k=0}^{n-1} \beta_k , \sum_{k=0}^{n} \beta_k \Big] .
\end{align*}
In~\cite{PAMS_AB_AF_MM} it is proven that this system contains a non-commutative version of an instance of discrete Painlev\'e~IV equation.

\section*{Acknowledgements}

AB acknowledges Centro de Matem\'atica da Universidade de Coimbra (CMUC) -- UID/MAT/00324/2020, funded by the Portuguese Government through FCT/MEC and co-funded by the European Regional Development Fund through the Partnership Agreement PT2020.
 
AFM and AF thanks CIDMA Center for Research and Development in Mathematics and Applications (University of Aveiro) and the Portuguese Foundation for Science and Technology (FCT) within project UIDB/04106/\linebreak2020 and UIDP/04106/2020.
 
MM was partially supported by the Spanish ``Agencia Estatal de Investigaci\'on'' research project [PGC2018-096504-B-C33], \emph{Ortogonalidad y~Apro\-ximaci\'on: Teor\'\i a y Aplicaciones en F\'\i sica Mate\-m\'a\-tica} and research project [PID2021- 122154NB-I00], \emph{Ortogonalidad y aproximaci\'on con aplicaciones en machine learning y teor\'\i a de la probabilidad}.

%
%

\begin{thebibliography}{99.}%
%
%


\bibitem{nuevo}
C. \'Alvarez--Fern\'andez, G. Ariznabarreta, J. C. Garc\'\i a--Ardila, M. Ma\~nas, \& F. Marcell\'an,
\emph{Christoffel transformations for matrix orthogonal polynomials in the real line and the non-Abelian 2D Toda lattice hierarchy}, 
International Mathematics Research Notices,
\textbf{2017}(5) (2017) 1285--1341.




\bibitem{nikishin}
A. I. Aptekarev \& E. M. Nikishin,
\emph{The scattering problem for a discrete Sturm-Liouville operator},
Math. USSR, Sb.,
49(2) (1984) 325--355.



\bibitem{Berezanskii}
Ju. M. Berezanskii,
\emph{Expansions in eigenfunctions of selfadjoint operators},
AMS Providence,~1968.


\bibitem{BAFMM}
A. Branquinho, A. Foulqui\'e-Moreno, \& M. Ma\~nas,
\emph{Matrix biorthogonal polynomials: Eigenvalue problems and non-abelian discrete Painlev\'e equations: A~Riemann--Hilbert problem perspective},
Journal of Mathematical Analysis and Applications,
494(2) (2021) 124605.



\bibitem{BFAM}
A. Branquinho, A. F. Moreno, A. Fradi, \& M. Ma\~nas,
\emph{Riemann--Hilbert Problem for the Matrix Laguerre Biorthogonal Polynomials: The Matrix Discrete Painlev\'e~IV},
Mathematics,
\textbf{10}(8) (2022) 1205,
Doi: https://doi.org/10.3390/math10081205.



\bibitem{PAMS_AB_AF_MM}
A. Branquinho, A. Foulqui\'e Moreno, A. Fradi, \& M. Ma\~nas,
\emph{Matrix Jacobi biorthogonal polynomials via Riemann--Hilbert problem},
Proceedings of the American Mathematical Society, Doi: https://doi.org/10.1090/proc/16431.



\bibitem{McCc} 
C. Calder\'on \& M. Castro,
\emph{Structural Formulas for Matrix-Valued Orthogonal Polynomials Related to $2 \times 2$ Hypergeometric Operator},
Bulletin of the Malaysian Mathematical Sciences Society,
\textbf{45}(2), 697--726.



\bibitem{McAg}
M. Castro \& F. A. Gr\"unbaum,
\emph{Orthogonal matrix polynomials satisfying first order differential equations: a collection of instructive examples},
Journal of Nonlinear Mathematical Physics,
Volume 12, Supplement 2 (2005), 63--76.


\bibitem{McAg2} 
M. Castro \& F. A. Gr\"unbaum,
\emph{The algebra of differential operators associated to a family of matrix-valued orthogonal polynomials: five instructive examples},
International Mathematics Research Notices, \textbf{2006} 47602.



\bibitem{duran2004} 
A. J. Dur\'an \& F. A. Gr\"unbaum, 
\emph{Orthogonal matrix polynomials satisfying second order differential equations}, 
International Mathematics Research Notices,
 \textbf{10} (2004) 461--484.




\bibitem{gakhov}
F. D. Gakhov, \emph{Boundary Value Problems},
New York: Dover Publications, Inc., 1990.



\bibitem{geronimo}
J. S. Geronimo, 
\emph{Scattering theory and matrix orthogonal polynomials on the real line}, 
Circuits Systems Signal Process,
1 (1982) 471--495.




\bibitem{grummmm}
F. A. Gr\"unbaum,
\emph{Matrix valued Jacobi polynomials},
Bull. Sci. Math.,
127(3):207--214, 2003


\bibitem{grummmmpatirao1}
F. A. Gr\"unbaum, I. Pacharoni, \& J. Tirao,
\emph{A matrix-valued solution to Bochner's problem},
J. Phys. A,
34(48):10647--10656, 2001.


\bibitem{AgPJ} 
F. A. Gr\"unbaum, I. Pacharoni, \& J. Tirao,
\emph{Matrix valued orthogonal polynomials of the Jacobi type},
Indagationes Mathematicae,
\textbf{14}(3-4) (2003) 353--366.




\bibitem{Krein1} 
M. G. Krein,
\emph{Infinite J-matrices and a matrix moment problem}, 
Dokl. Akad. Nauk SSSR,
 69 (2) (1949) 125--128.

\bibitem{Krein2} 
M. G. Krein,
\emph{Fundamental aspects of the representation theory of hermitian operators with deficiency index $(m, m)$}, 
AMS Translations, 
Series 2, vol. 97, 
Providence, Rhode Island, 1971, 
75--143.





\bibitem{SvA}
A. Sinap \& W. Van Assche,
\emph{Orthogonal matrix polynomials and applications},
Journal of Computational and Applied Mathematics,
\textbf{66}(1-2) (1996) 27--52.


\bibitem{Szego}
G. Szeg\H{o},
\emph{Orthogonal Polynomials},
American Mathematical Society,~1939.


\bibitem{VAssche} 
W. Van Assche, 
\emph{Discrete Painlev\'e equations for recurrence coefficients of orthogonal polynomials},
Proceedings of the International Conference on Difference Equations, Special Functions and Orthogonal Polynomials,
687--725, World Scientific (2007).


\bibitem{wasow}
W. Wasow, 
\emph{Asymptotic expansions for ordinary differential equations}, 
Pure and Applied Mathematics,
Vol.~XIV, Interscience publishers, New York,~1965.


\end{thebibliography}
%

\end{document}